\documentclass{asl}
\usepackage[margin=1.5in]{geometry}
\usepackage{graphicx}
\usepackage{enumitem}
\usepackage[scr]{rsfso}
\usepackage{xspace}
\usepackage{tikz}
\usetikzlibrary{arrows.meta} \tikzset{>={Latex[width=3mm,length=2mm]}}
\usepackage{subcaption}
\usepackage[title]{appendix}
\usepackage[backend=biber,style=authoryear,natbib=true]{biblatex}
\newcommand\Label[1]{&\refstepcounter{equation}(\theequation)\ltx@label{#1}&}

\bibliography{Frege.bib}

\theoremstyle{plain}% default

\newtheorem{thm}{Theorem}[section]
\newtheorem{lem}[thm]{Lemma}

\theoremstyle{definition}
\newtheorem{definition}[thm]{Definition}
\newtheorem{exmp}[thm]{Example}

\theoremstyle{remark}
\newtheorem{rem}[thm]{Remark}

\newtheorem{innercustomgeneric}{\customgenericname}
\providecommand{\customgenericname}{}
\newcommand{\newcustomtheorem}[2]{%
  \newenvironment{#1}[1]
  {%
   \renewcommand\customgenericname{#2}%
   \renewcommand\theinnercustomgeneric{##1}%
   \innercustomgeneric \it
  }
  {\endinnercustomgeneric} 
}

\newcustomtheorem{cthm}{{\sc Theorem}}

\newcommand{\modelm}{\mathcal{M}}
\newcommand{\powerset}{\mathscr{P}}

\newcommand{\lang}{\mathcal{L}}

\newcommand{\set}[1]{\{ #1 \}}

\newcommand{\PAT}{\mathsf{PA}^2\xspace}
\newcommand{\PAO}{\mathsf{PA}^1\xspace}
\newcommand{\TAO}{\mathsf{TA}^1\xspace}
\newcommand{\TAT}{\mathsf{TA}^2\xspace}
\newcommand{\Q}{\mathsf{Q}\xspace}
\newcommand{\HP}{\ensuremath{\mathsf{HP}^2}\xspace}
\newcommand{\ACA}{\mathsf{ACA}_0\xspace}
\newcommand{\PICA}{\Pi^1_1\mbox{-}\mathsf{CA}_0\xspace}
\newcommand{\PICAk}[1]{\Pi^1_{#1}\mbox{-}\mathsf{CA}_0\xspace}
\newcommand{\PI}{\mathsf{I_{PI}}\xspace}
\newcommand{\EPI}{\mathsf{E_{PI}}\xspace}

\newcommand{\nn}{\mathbb{N}}
\newcommand{\momodels}[2]{\modelm,#1\vDash}

\title{The Potential in Frege's Theorem}

\author{Will Stafford}
\revauthor{Stafford, Will}
\address{Department of Logic and Philosophy of Science\\ University of California, Irvine\\ Irvine, 92617 CA, USA}
\email{will.stafford@uci.edu}
\thanks{I would like to thank the audience at the Logic Colloquium 2016 in Leeds and the UCI Logic Seminar 2017 for their questions and comments, and Tim Button, Jeremy Heis, Richard Mendelsohn, Stella Moon, Sean Walsh, and Kai Wehmeier for their helpful feedback.}
\begin{document}
%%%%%%%%%%%%%%%%%%%%%%%%%%%%%%%%
%%%%%%%%%%%%%%%%%%%%%%%%%%%%%%%%
%%%%%%%%%%%%%%%%%%%%%%%%%%%%%%%%
\maketitle
%%%%%%%%%%%%%%%%%%%%%%%%%%%%%%%%
%%%%%%%%%%%%%%%%%%%%%%%%%%%%%%%%
%%%%%%%%%%%%%%%%%%%%%%%%%%%%%%%%
\begin{abstract}
    Is a logicist bound to the claim that as a matter of analytic truth there is an actual infinity of objects? If Hume's Principle is analytic then in the standard setting the answer appears to be yes. Hodes's work pointed to a way out by offering a modal picture in which only a potential infinity was posited. However, this project was abandoned due to apparent failures of cross-world predication.  We re-explore this idea and discover that in the setting of the potential infinite one can interpret first-order Peano arithmetic, but not second-order Peano arithmetic.  We conclude that in order for the logicist to weaken the metaphysically loaded claim of necessary actual infinities, they must also weaken the mathematics they recover.
\end{abstract}
%%%%%%%%%%%%%%%%%%%%%%%%%%%%%%%%
%%%%%%%%%%%%%%%%%%%%%%%%%%%%%%%%
%%%%%%%%%%%%%%%%%%%%%%%%%%%%%%%%
\tableofcontents
%%%%%%%%%%%%%%%%%%%%%%%%%%%%%%%%
%%%%%%%%%%%%%%%%%%%%%%%%%%%%%%%%
%%%%%%%%%%%%%%%%%%%%%%%%%%%%%%%%
\section{Introduction}\label{sec:intro}
%%%%%%%%%%%%%%%%%%%%%%%%%%%%%%%%
%%%%%%%%%%%%%%%%%%%%%%%%%%%%%%%%
%%%%%%%%%%%%%%%%%%%%%%%%%%%%%%%%
\subsection{Potentially Infinite Models}
%%%%%%%%%%%%%%%%%%%%%%%%%%%%%%%%
%%%%%%%%%%%%%%%%%%%%%%%%%%%%%%%%
%%%%%%%%%%%%%%%%%%%%%%%%%%%%%%%%
\noindent  In the non-modal setting, Frege (\citeyear{Frege1903}; \cite{heck1993}) essentially proved that second-order Peano arithmetic, $\PAT$, is interpretable in the theory \HP, which consists of the Second-order Comprehension Schema and Hume's Principle:
%EQUATION
%%%%%%%%%%%%%%%%%%%%%%%%%%%%%%%%
\begin{equation}\label{hp}\tag{HP}
\forall X,Y(\#X=\#Y\Leftrightarrow\exists \ \text{\textit{bijection}} \ f:X\rightarrow Y).
\end{equation}
%%%%%%%%%%%%%%%%%%%%%%%%%%%%%%%%

\noindent Hume's Principle characterises the cardinality operator $\#$, read `the number of' or `octothorpe', as a type-lowering function that takes equinumerous second-order objects to the same first-order object.   This definition can be motivated in the finite case by examples such as checking one has the same number of knives and forks by setting them out in pairs. Formally, Frege's result is:

%THEOREM
%%%%%%%%%%%%%%%%%%%%%%%%%%%%%%%%
\begin{thm}[Frege's Theorem]\label{ft}
There is a translation from the language of $\PAT$ to the language of \HP that interprets $\PAT$ in $\HP$.
\end{thm}
%%%%%%%%%%%%%%%%%%%%%%%%%%%%%%%%

\noindent The formal definition of the theories mentioned here can be found in Appendix~\ref{ap:theories}. Frege's Theorem has traditionally been regarded as philosophically important because it is supposed to show that we can derive all arithmetical theorems from an epistemically innocent system. This requires that Hume's Principle is analytic.  However, on the usual semantics, Hume's Principle is only true on domains with at least a countable infinity of objects.  This commits logicists like Frege to the analytic existence of an actual infinity of objects (\cite[pp. 199, 213, 233]{Boolos1998a}; \cite[pp. 20, 292, 309]{Hale2001}; \cite[p.~7]{Cook2007}).  

A commitment to a \emph{potential} infinity, in contrast, isn't a commitment to how many things there actually are, just how many are possible.  This is a much safer area in which to make analytic claims.  
Here we show that some but not all of the mathematics of the actual infinite is recoverable in the setting of the potential infinite.  And so, to avoid problematic ontological commitments the logicist must also weaken the mathematics they recover.  

To do this we must decide how to represent Hume's Principle.  Below we will define `the number of' operator $\#$ in a semantic manner.  However, we are convinced that this is simply a convenience and we can think of our models as defining $\#$ as satisfying Hume's Principle with the additional criteria that this function is rigid across worlds.  An axiomatization would consist of the following modification of Hume's Principle: 
%EQUATION
%%%%%%%%%%%%%%%%%%%%%%%%%%%%%%%%
$$ \Box\forall X,Y(\#X=\#Y\Leftrightarrow\exists \ \text{\textit{bijection}} \ f:X\rightarrow Y), $$
%%%%%%%%%%%%%%%%%%%%%%%%%%%%%%%%

\noindent plus a principle to rigidify the $\#$ operator.  This would require working in a hybrid modal logic where worlds could be saved and recalled such as \textcite[370]{Williamson2013-df}.\footnote{For those familiar with hybrid systems the axioms needed is $\uparrow \Box \forall \; X,y \; \downarrow [ \#X=y \rightarrow \Box \#X=y].$ However, this will not play a role in what follows.} However, we leave the details of this approach for future work.  As the modification is so minimal, the move to the potentially infinite doesn't undermine the justifications offered for Hume's Principle.  The syntactic priority thesis can still be argued for as we can identify the behaviour of terms in a modal setting as well as in a non modal setting.  Similarly if we think that abstraction principles offer implicit definitions then this justification works as well in the modal setting.

The rigidity of the octothorpe is important for the success of the project here.  However, by assuming that it is rigid we are presuming that `the number of' operator is rigid.  Whether this is the case in natural language is an empirical question (e.g.\@ \cite{Stanley1997-ub}).  We do not address this issue here, but two things are worth noting.  First the question of the rigidity of `the number of' is not the same question as e.g.\@ whether the number of planets varies between worlds.  This is because we do not apply the operator to predicates but rather to sets which do not vary their membership across worlds.  The second is that this setting does rule out the possibility of multiple different number structures in the different worlds, e.g.\@ the numbers being von Neumann ordinals in one world and Zermelo ordinals in another.   This means that a certain kind of referential indeterminacy which has a prominent place in philosophy of mathematics cannot be addressed in this setting as we have presumed against it (\cite{Benacerraf1965-pa}; \cite[ch.~2]{Button2018-in}).

To set up our result, we define a set of second-order Kripke models, which we will call \textit{potentially infinite models}. This idea comes from \textcite[379]{Hodes1990}, although he does not place exactly these constraints on the accessibility relation.  We want the models to be nearly linear sequences of worlds (if there are two worlds neither of which accesses the other, there is a third world they both access), where later worlds are possible from the perspective of earlier worlds but not the other way around.  Each of these worlds should contain only a finite number of objects as we are assuming actual infinities are impossible, and the number of objects should increase from one world to the next.  Each world will have its own second-order domain, which as the worlds are finite, will be the full powerset.  The octothorpe will implement Hume's Principle by taking sets of the same cardinality to a unique object and this object will not change from one world to the next. We define the models formally as follows: 

%DEFINITION
%%%%%%%%%%%%%%%%%%%%%%%%%%%%%%%%
\begin{definition}\label{jo}
A \emph{potentially infinite (PI) model} is a quadruple $\mathcal{M}=\langle W,R,D,I \rangle$ in the modal signature with second-order quantification and with $\#$ and $\mathbf{a}$ as the only non-logical symbols, such that the following conditions are met:
    %LIST
    %%%%%%%%%%%%%%%%%%%%%%%%%%%%%%%%
    \begin{enumerate}[leftmargin=1cm]
        \item[\ref{jo}.1.]{ $W$ is countably infinite and $R$ is a directed partial order,\footnote{An order $R$ is directed if for all $w,s\in W$ there exists an $t\in W$ such that $R(w,t)$ and $R(s,t)$.}} 
                
        \item[\ref{jo}.2.]{ the first-order domain of $w$, written $D(w)$, is non-empty and finite for all $w\in W$, 
        
        \item[\ref{jo}.3.] for each $n\geq 1$, the range of the second-order $n$-ary relational quantifiers at $w$ is $\powerset(D(w)^n)$ consisting of all subsets of the $n$-th Cartesian power $(D(w))^n$ of $D(w)$,}
            
        \item[\ref{jo}.4.] {if $w,s\in W$ such that $R(w,s)$ and $w\neq s$, then $D(w)\subsetneq D(s),$ }
                
        \item[\ref{jo}.5.] the function $\mathbf{a}:\omega\rightarrow D$ (where $D$ is $\bigcup_{w\in W}D(w)$) assigns to each number $n$ a distinct element $\mathbf{a}_n$ in one of the first-order domains, and for all $w\in W$,  the cardinality of $X$ is $n$ if and only if $\#X=\mathbf{a}_n$ at $w$.  More formally, for $\#$ and all $w$ the interpretation function is defined as follows:  $I(\#,w)=\set{\langle X,\mathbf{a}_{|X|}\rangle\mid\exists s\in W \; X\in\powerset (D(s))}$.
    \end{enumerate}
    %%%%%%%%%%%%%%%%%%%%%%%%%%%%%%%%
\end{definition}
%%%%%%%%%%%%%%%%%%%%%%%%%%%%%%%%

%REMARK
%%%%%%%%%%%%%%%%%%%%%%%%%%%%%%%%
\begin{rem} Three brief remarks on this definition:

First, conditions \ref{jo}.1-4 define a PI model as a directed partial order of ever-increasing finite domains.  This means that if we have several objects existing in different possible worlds we can always move to a world where they all exist. 

Second, condition \ref{jo}.5 defines the cardinality operator $\#$ using metatheoretic cardinality $|X|$.  It is sufficient for Hume's Principle to hold that $\#$ picks-out cardinality, and so condition \ref{jo}.5 ensures that all potentially infinite models are models of Hume's Principle. One reason we need $\powerset(D(w)^2)$ from \ref{jo}.3 is because the quantifier over graphs of functions in Hume's Principle ranges over this set.  

Third, condition \ref{jo}.5 also ensures that the interpretation of the octothorpe is rigid.  That is, the octothorpe is interpreted as the same relation at every world.  Because of this nothing will be lost if we write $\#X=x$ and don't specify the world of evaluation.  In fact, while we define $\#X$ using the $\mathbf{a}_i$'s, we could have instead simply defined it as rigid and satisfying Hume's Principle and this along with directedness would ensure the $\mathbf{a}_i$'s exist.
\end{rem}
%%%%%%%%%%%%%%%%%%%%%%%%%%%%%%%%

%FIGURE
%%%%%%%%%%%%%%%%%%%%%%%%%%%%%%%%
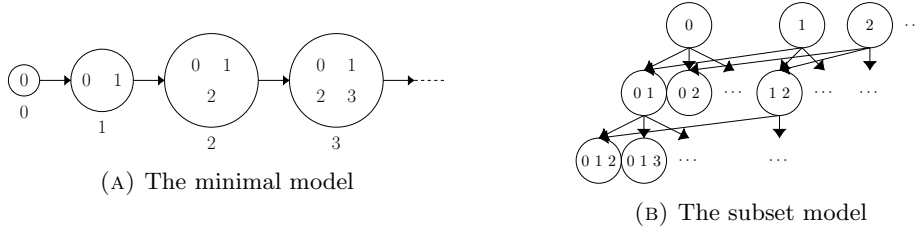
\begin{figure}
%CENTRING FIGURE
\centering
    %SUBFIGURE 1
    %%%%%%%%%%%%%%%%%%%%%%%%%%%%%%%%
    \begin{subfigure}{.49\textwidth}
    %CENTRING SUBFIGURE 1
    \centering
    %RESIZING SUBFIGURE 1
    \resizebox{.85\textwidth}{!}{%
        %TIKZ SUBFIGURE 1
        \begin{tikzpicture}
    \draw (0,1) circle [radius=0.5];;
    \draw (2.5,1) circle [radius=1];;
    \draw (6,1) circle [radius=1.5];;
    \draw (10,1) circle [radius=1.5];;
    
    \draw [->] (0.5,1) -- (1.5,1);
    \draw [->] (3.5,1) -- (4.5,1);
    \draw [->] (7.5,1) -- (8.5,1);
    \draw [->] (11.5,1) -- (12.5,1);
    \draw [dashed] (12.5,1) -- (13.5,1);
    
    \node at (0,0) {\huge 0};
    \node at (2.5,-0.5) {\huge 1};
    \node at (6,-1) {\huge 2};
    \node at (10,-1) {\huge 3};
    
    \node at (0,1) {\huge 0};
    
    \node at (2,1) {\huge 0};
    \node at (3,1) {\huge 1};
    
    \node at (5.5,1.5) {\huge 0};
    \node at (6.5,1.5) {\huge 1};
    \node at (6,0.5) {\huge 2};
    
    \node at (9.5,1.5) {\huge 0};
    \node at (10.5,1.5) {\huge 1};
    \node at (9.5,0.5) {\huge 2};
    \node at (10.5,0.5) {\huge 3};
    \end{tikzpicture}}
    %CAPTION SUBFIGURE 1
    \caption{The minimal model}
    %LABLE SUBFIGURE 1
    \label{fig:minimal}
    \end{subfigure}
    %SUBFIGURE 2
    \begin{subfigure}{.49\textwidth}
    %CENTERING SUBFIGURE 2
    \centering
    %RESIZING SUBFIGURE 2
    \resizebox{.7\textwidth}{!}{
        %TIKZ SUBFIGURE 2
        \begin{tikzpicture}
    \draw (0,1) circle [radius=0.5];;
    \draw (2.5,1) circle [radius=0.5];;
    \draw (4,1) circle [radius=0.5];;
    
    \node at (0,1) {0};
    \node at (2.5,1) {1};
    \node at (4,1) {2};
    \node at (5,1) {$\cdots$};
    
    \draw [->] (0,0.5) -- (-1,0);
    \draw [->] (0,0.5) -- (1,0);
    \draw [->] (0,0.5) -- (0,0);
    \draw [->] (2.5,0.5) -- (-1,0);
    \draw [->] (2.5,0.5) -- (2,0);
    \draw [->] (2.5,0.5) -- (3,0);
    \draw [->] (4,0.5) -- (0,0);
    \draw [->] (4,0.5) -- (2,0);
    \draw [->] (4,0.5) -- (4,0);
    
    \draw (-1,-0.5) circle [radius=0.5];;
    \draw (0,-0.5) circle [radius=0.5];;
    \draw (2,-0.5) circle [radius=0.5];;
   
    \node at (-1,-0.5) {0 1};
    \node at (1,-0.5) {$\cdots$};
    \node at (0,-0.5) {0 2};
    \node at (2,-0.5) {1 2};
    \node at (3,-0.5) {$\cdots$};
    \node at (4,-0.5) {$\cdots$};
    
    \draw [->] (-1,-1) -- (-2,-1.5);
    \draw [->] (-1,-1) -- (0,-1.5);
    \draw [->] (-1,-1) -- (-1,-1.5);
    
    \draw (-2,-2) circle [radius=0.5];;
    \draw (-1,-2) circle [radius=0.5];;
    
    \node at (-2,-2) {0 1 2};
    \node at (0,-2) {$\cdots$};
    \node at (-1,-2) {0 1 3};
    
    \draw [->] (2,-1) -- (2,-1.5);
    \draw [->] (2,-1) -- (-2,-1.5);
    
    \node at (2,-2) {$\cdots$};
  
    \end{tikzpicture}}
    %CAPTION SUBFIGURE 2
    \caption{The subset model}
    %LABEL SUBFIGURE 2
    \label{fig:subset}
    \end{subfigure}
    %CAPTION FIGURE
    \caption{Examples of potentially infinite models}
\end{figure}
%%%%%%%%%%%%%%%%%%%%%%%%%%%%%%%%

\noindent This definition can obscure the simplicity of the idea here, as such it helps to give several examples.  The simplest potentially infinite model we can construct is the following:

%EXAMPLE
%%%%%%%%%%%%%%%%%%%%%%%%%%%%%%%%
\begin{exmp}\label{exmp:min}
The \emph{minimal} potentially infinite model is $(\omega,\leq, D,I)$ where $D(\mathbf{n})=\set{\mathbf{0},\dots,\mathbf{n}}$ and the interpretation function $I$ interprets octothorpe as cardinality in the metalanguage.\footnote{I will use bold face numbers for the numbers in the metalanguage.} That is, $I(\#,w)(X)$ $=\mathbf{n}$ if and only if $|X|=\mathbf{n}$.  The minimal model is illustrated in Figure~\ref{fig:minimal}.  When working with such a model we see that a number can be missing from a world {\it even} if a set of that cardinality is present.  So $I(\#,\mathbf{1})(\set{\mathbf{0}})=\mathbf{1}$ and $\mathbf{1}\in D(\mathbf{1})$, but $I(\#,\mathbf{1})(\set{\mathbf{0},\mathbf{1}})=\mathbf{2}$ and $\mathbf{2}\notin D(\mathbf{1})$ even though $\set{\mathbf{0},\mathbf{1}}\subseteq D(\mathbf{1})$.
\end{exmp}
%%%%%%%%%%%%%%%%%%%%%%%%%%%%%%%%

\noindent A less simple but similarly elementary model makes use of the non-empty finite subsets of the natural numbers.  This model helps illustrate a non-linear $R$ relation:

%EXAMPLE
%%%%%%%%%%%%%%%%%%%%%%%%%%%%%%%%
\begin{exmp}
Let the \emph{subset} model be $(\powerset(\omega)^{<\omega}-\set{\varnothing},\subseteq, D,I)$ where $D(X)=X$ and again the octothorpe is cardinality.  The subset model is illustrated in Figure~\ref{fig:subset}.  Note that if we have worlds $X_0,\dots,X_n$ we can always find an accessible world whose domain is $\bigcup^n_{i=0}X_i$.  For example, $\set{\mathbf{0},\mathbf{1}},\set{\mathbf{3}},\set{\mathbf{100},\dots,\mathbf{200}}$ are all finite subsets of the natural numbers, none of which access each other, however, their union $\set{\mathbf{0},\mathbf{1},\mathbf{3},\mathbf{100},\dots,\mathbf{200}}$ is also a world, which they all access.
\end{exmp}
%%%%%%%%%%%%%%%%%%%%%%%%%%%%%%%%

\noindent It is easy to generate unintended models from these two cases.  Using the minimal model, for example, we can define the \textbf{3}-\textbf{0} swap model:

%EXAMPLE
%%%%%%%%%%%%%%%%%%%%%%%%%%%%%%%%
\begin{exmp}
The \emph{\textbf{3}-\textbf{0} swap} model takes \textbf{0} and \textbf{3} in the domain of the minimal model and switches them around.  So $D(\mathbf{0})=\set{\mathbf{3}},$ $D(\mathbf{1})=\set{\mathbf{3},\mathbf{1}},$ $D(\mathbf{2})=\set{\mathbf{3},\mathbf{1},\mathbf{2}},$ $D(\mathbf{3})=\set{\mathbf{3},\mathbf{1},\mathbf{2},\mathbf{0}}$ and then for all $\mathbf{n}\geq \mathbf{3}$, we have that $D(\mathbf{n})$ exactly as it is in the minimal model. 
\end{exmp}
%%%%%%%%%%%%%%%%%%%%%%%%%%%%%%%%

\noindent These models should help illustrate the intuition behind the potentially infinite models.  They will also be helpful when we need counterexamples to claims later in the paper.

We can now define satisfaction for potentially infinite models using a standard semantics for quantified modal logic, such as in \textcite{Fitting1998}. Three things to note first: (1) Our quantifiers are actualist, but free variables may be assigned to objects in any world. (2) Set variables are interpreted rigidly across worlds.  That is the membership of a set doesn't change depending on the world.  (3) To simplify the notation, instead of variable assignments, we work as though we had a rigid name for every object in the models.   Recall that $\momodels{w}{g}\varphi$  means that given any replacement of free variables with the added constants we evaluate $\varphi$ as true in $\mathcal{M}$ at world $w$.   With this in place, the notion of potentially infinite models induces a natural validity relation, which we define as follows:

%DEFINITION
%%%%%%%%%%%%%%%%%%%%%%%%%%%%%%%%
\begin{definition}\label{pivalid}
We say that $\varphi$ is \textit{true in all potentially infinite models}, or $\vDash_{\mathsf{PI}}\varphi$, if for all potentially infinite models $\mathcal{M}$ and worlds $w\in W$  we have $\momodels{w}{g}\varphi$.   We define $\varphi\vDash_{\mathsf{PI}}\psi$ as for all models $\mathcal{M}$ and worlds $w\in W$, if $\momodels{w}{g}\varphi$ then $\momodels{w}{g}\psi$.  
\end{definition}
%%%%%%%%%%%%%%%%%%%%%%%%%%%%%%%%

\noindent The consequence relation here is defined locally rather than globally (\cite[21]{Fitting1998}).  This is because the deduction theorem holds for the local consequence relation but not the global one (\cite[23]{Fitting1998}).

%%%%%%%%%%%%%%%%%%%%%%%%%%%%%%%%
%%%%%%%%%%%%%%%%%%%%%%%%%%%%%%%%
%%%%%%%%%%%%%%%%%%%%%%%%%%%%%%%%
\subsection{Main Results}\label{sec:mainresults}
%%%%%%%%%%%%%%%%%%%%%%%%%%%%%%%%
%%%%%%%%%%%%%%%%%%%%%%%%%%%%%%%%
%%%%%%%%%%%%%%%%%%%%%%%%%%%%%%%%

\noindent We will now state our two main results which together show that we can interpret the first-order theories of first-order Peano arithmetic $\PAO$ and first-order true arithmetic $\TAO$, but not the second-order theories of second-order Peano arithmetic $\PAT$ and second-order true arithmetic $\TAT$, in theories defined in terms of potentially infinite models.  A deductive theory for second-order modal logic with rigid operators would be unwieldy and the complications caused by 
it would be likely to obscure the insights provided by the Kripke semantics. Hence, we leave development of a deductive theory for future work. We can define a theory just in terms of the potentially infinite models.  This theory will be stronger than anything we could produce deductively because it does not admit nonstandard models of the natural numbers.  Because of this we will call it the external theory of the potentially infinite or ${\EPI}$:
%EQUATION
%%%%%%%%%%%%%%%%%%%%%%%%%%%%%%%%
\begin{equation}\label{equ:etpi}
    \EPI=\set{\varphi\mid\; \vDash_{\mathsf{PI}}\varphi}.
\end{equation}
%%%%%%%%%%%%%%%%%%%%%%%%%%%%%%%%

To capture something closer to what can be deduced from the models we need to use the model-theoretic validity relation defined above, relativised to a weak metatheory. The theory $\ACA$ is a subsystem of $\PAT$ which only has comprehension for first-order formulas.  More information about this theory can be found in Appendix~\ref{ap:theories}. Since we can code finite sets of natural numbers as natural numbers in $\ACA$, we can define the property of being a potentially infinite  model in this theory, along with the associated validity notion $\vDash_{\mathsf{PI}}$.  This gives us the internal theory of the potentially infinite or $\PI$:
%EQUATION
%%%%%%%%%%%%%%%%%%%%%%%%%%%%%%%%
\begin{equation}\label{equ:tpi}
    \PI=\set{\varphi\mid \ACA\vdash\text{`$\vDash_{\mathsf{PI}}\varphi$'}}.
\end{equation} 
%%%%%%%%%%%%%%%%%%%%%%%%%%%%%%%%

\noindent Intuitively, this theory is every formula that can be proven valid on potentially infinite models, given the weakest metatheory that can formalise the models. A full definition is given in Appendix~\ref{ap:piintoapa0}.\footnote{We picked the weakest theory because we are interested in what is deducible from PI models and if we strengthen the metatheory $\PI$ will be strengthened in ways that reflect what the metatheory thinks about finite sets (which can code consistency statements).}  The definition of interpretation is traditionally restricted to theories in the same logic, whereas in this setting $\EPI$ and $\PI$ are theories in second-order modal logic but $\PAO$, $\PAT$, $\TAO$, and $\TAT$ aren't modal theories. So, to state and prove our main results we need a more general notion of \textit{generalised translation} and \textit{interpretation} which captures those interpretations which involve not just different theories but different logics.  This is defined in section~\ref{sec:interpertation}. Our first main result is:

%THEOREM
%%%%%%%%%%%%%%%%%%%%%%%%%%%%%%%%
\begin{thm}\label{thm:main} \vspace{0cm}
%LIST
%%%%%%%%%%%%%%%%%%%%%%%%%%%%%%%%
\begin{enumerate}
    \item[(i)] There is a generalised translation from the language of $\PAO$ to the second-order modal language with octothorpe that interprets $\TAO$ in $\EPI$.
    \item[(ii)] There is a generalised translation from the language of $\PAO$ to the second-order modal language with octothorpe that interprets $\PAO$ in $\PI$. Further, this is a $\PAO$-verifiable generalised interpretation.
\end{enumerate}
%%%%%%%%%%%%%%%%%%%%%%%%%%%%%%%%
\end{thm}
%%%%%%%%%%%%%%%%%%%%%%%%%%%%%%%%

\noindent This result is proven in Section~\ref{sec:interpertation}.  The translation used is based on one offered by \textcite{Linnebo2013} in the setting of modal set theory.  The key difference, compared with the standard notion of translation, is that ``for all'' is translated as ``necessarily for all'' and, similarly, ``there is'' is translated as ``possibly there is.''  

The first theorem shows that the PI models capture a significant amount of mathematics.  However, we cannot strengthen the result to second-order theories of arithmetic as our second main theorem shows: 

%THEOREM
%%%%%%%%%%%%%%%%%%%%%%%%%%%%%%%%
\begin{thm}\label{thm:main2}\vspace{0cm}
%LIST
%%%%%%%%%%%%%%%%%%%%%%%%%%%%%%%%
\begin{enumerate}
    \item[(i)] There is no generalised translation from the language of $\PAT$ to the second-order modal language with octothorpe that interprets $\TAT$ in $\EPI$.
    \item[(ii)] There is no  generalised translation from the language of $\PAT$ to the second-order modal language with octothorpe that $\PAT$-verifiably interprets $\PAT$ in $\PI$.
\end{enumerate}
%%%%%%%%%%%%%%%%%%%%%%%%%%%%%%%%
\end{thm}
%%%%%%%%%%%%%%%%%%%%%%%%%%%%%%%%

\noindent For both $\EPI$ and $\PI$, the results follow from the fact that PI models are $\Pi^1_1$ definable.  And this follows because all of the worlds are finite.  Because of this, PI models are representable in reasonably weak theories of second-order arithmetic.  But then limitive results about what theories can represent about themselves will stop theories that can represent $\EPI$ and $\PI$ being interpretable into $\EPI$ and $\PI$.

These results are important because they show that less mathematics is analytic on the philosophical perspective which motivates the potentially infinite models than on the traditional perspective.  The external theory cannot recover $\TAT$ but only $\TAO$.  And the internal theory cannot recover $\PAT$ but only $\PAO$. Further, $\PAT$ has traditionally been the target of Fregean interpretation results as it allows for the recovery of analysis and much of mathematics.\footnote{\textcite[238 n26]{Demopoulos1994} points out that Frege often uses arithmetic when he means something broader including analysis.
} Analysis can be coded in second-order Peano arithmetic, as real numbers can be coded as sets of rationals, which in turn can be coded as naturals.  This means that Frege's theorem already accounts for a larger expanse of mathematics than it might first appear. If we try to avoid the claim that it is analytic that there are actually infinitely many objects, however, it then seems we will not have managed to recover as much mathematics.  If we are looking to show that mathematics is analytic, we have moved further from our goal.  

However, we have still captured a substantial chunk of our most frequently used mathematics.  \textcite[613]{Feferman2005-nf} has argued that all scientifically applicable analysis can be developed in $\PAO$ or a conservative extension of it.\footnote{For example, ``By the fact of the proof-theoretical reduction of $W$ to [$\PAO$], the only ontology it commits one to is that which justifies acceptance of [$\PAO$]." (\cite[613]{Feferman2005-nf})  Feferman works in a system $W$ which contains types for the naturals, the cross product and partial functions.  The full classical analysis of continuous functions can be carried out in $W$. (\cite[611]{Feferman2005-nf})} If this is correct then we can still recover the mathematics for which an explication of its truth is most philosophically fruitful, namely the mathematics which we rely on when we act in the world.  One might wonder why a logicist would care about whether or not the mathematics recovered is used.  But it seems we should keep an open mind to different parts of mathematics being justified in different ways.  Maybe something as fundamental as first-order arithmetic turns out to be analytic, but it seems unlikely that the same is true of the higher reaches of set theory.  With this in mind, it should not be damaging that not all mathematics turns out to be analytic.

\subsection{A Diversity of Modal Logicisms}

\noindent The idea of using the potentially infinite as a foundation of logicism has a pedigree in the work of Putnam and Hodes, and more recent work on modal foundations of mathematics and on variants of Frege's theorem in different logics.  Putnam suggested that by accepting a modal picture of mathematics we could avoid being Platonists about the numbers or committing to how many objects there actually are.  This is stated most clearly when he writes:

%QUOTE
%%%%%%%%%%%%%%%%%%%%%%%%%%%%%%%%
\begin{quote}
    `Numbers exist'; but all this comes to, for mathematics anyway, is that (I) $\omega$-sequences are possible (mathematically speaking); and (2) there are necessary truths of the form `if $\alpha$ is an $\omega$-sequence, then\dots '[.] (\cite[11--12]{Putnam1967a}) 
\end{quote}
%%%%%%%%%%%%%%%%%%%%%%%%%%%%%%%%

\noindent Hodes took on this idea, but he was sceptical of the existence of actual infinities. He thought that `[a]rithmetic should be able to face boldly the dreadful chance that in the actual world there are only finitely many objects' (\cite[148]{Hodes1984}).  His solution made use of the idea of the potentially infinite rather than the actually infinite. He appealed to modality and in particular the modality that seems to be implicit in our concept of number: the idea that it is always possible to add 1~(\cite[378]{Hodes1990}). 

However, by 1990, Hodes concluded that the reduction of mathematics to higher-order modal logic had failed. Hodes describes the problem as follows:

%QUOTE
%%%%%%%%%%%%%%%%%%%%%%%%%%%%%%%%
\begin{quote}
The problem is simple: relative to [a model of Hume's Principle] for a type-0 variable $v$, $\Diamond(\exists v)(\underline{N}(v)\&\dots)$ ``moves us'' to other worlds $u$ and then has us seek a witnessing member of [the natural number in the model] in [the domain of $u$]; we may find one, but then have no way ``back'' to $w$ to see what hold [sic] for it there. (\cite[388]{Hodes1990}) 
\end{quote}
%%%%%%%%%%%%%%%%%%%%%%%%%%%%%%%%

\noindent So we might know that there possibly exists a number with a property, but in Hodes's system, we have no way of returning to our original world to use what we have found.  For example, if we find the number of a set in some world, we have no assurance that this number is available for us to talk about in the world the set came from. It is only known that it is the number of the set {\it in the world the number exists in}.  The difficulty identified here is with cross-world predication, which occurs when we want to say something about an object in one world and how it relates to objects in another world (\cite{Kocurek}).

In what follows we will show that the problem is not with cross-world predication \emph{per~se}. Both by working directly with the models, but also by allowing the octothorpe to be rigid, we can mimic some of the effects of cross-world predication. Yet in this setting we recover some but not all of the arithmetic recovered by Frege's theorem. Indeed, our main results, Theorems~\ref{thm:main}~and~\ref{thm:main2}, show that the situation is more complicated than Hodes suggested, and that a partial realisation of his project is possible. 

There are two recent trends in the study of logicism which this project is connected to.  First, \textcite{Studd2015} has suggested that the modal setting is an attractive one for the logicist because it would help to solve the {bad company objections}.  Unlike here, Studd's is concerned with inconsistent abstraction principles and in particular set abstraction.  This is interestingly connected to the na\"ive conception of set because one can think of the unrestricted set Comprehension Schema as similar in spirit to a modal version of Basic Law V. While work in this area goes back to \textcite{Parsons1983aa}, it has been pursued recently by Linnebo (\citeyear{Linnebo2013}; \citeyear{Linnebo2018}). Much of Linnebo's work has been on set theory.  The concerns there are very different from ours, as it make little sense in set theory to worry about the actual infinite not existing and set theory is generally treated in first-order logic.The work in this paper takes inspiration from the results presented in \textcite{Linnebo2013} and (\citeyear{Linnebo2018}) and makes use of a similar method of translating between the modal and non-modal setting. However, while the dynamic abstraction principles discussed by \textcite{Linnebo2018} resemble the behaviour of the \emph{number of} operator, his preferred abstraction principle for arithmetic is ordinal abstraction (\cite[Ch. 10.5]{Linnebo2018}), whereas in this paper we work with a modal version of Hume's Principle, a cardinality principle.  

Second, there has been a lot of recent work on whether Frege's Theorem still holds when the logic is modified in certain ways.  \textcite{Bell1999-BELFTI-2} and Shapiro and Linnebo (\citeyear{Shapiro2015-SHAFMB}) have shown that Frege's Theorem is available in the intuitionistic setting.  \textcite{Burgess2005} and \textcite{Walsh2016} found that a version of Frege's Theorem is possible in a certain predicative setting.  \textcite{Kim2015aa} proves a version of Frege's Theorem in a modal setting. This employs an axiomatised version of the `the number of $F$'s is $n$' as a binary relation, instead of the traditional type-lowering `number of' operator.   Kim recovers the axioms of PA but finds that a restricted version of $\HP$ holds.  The modality used is S5 and meant to represent logical possibility, not potentiality.  Because of this Kim's system does not have the same structure of our models, where the numbers slowly grow.  Closer in spirit to the work here is that on finite models of arithmetic by \textcite{Mostowski2001}.  There he considers initial sequences of the natural numbers and what holds over all such models.  These have a clear connection to the minimal model discussed above. \textcite{Urbaniak2016}  has taken Mostowski's models and worked with them in a modal setting.  They have shown that Le\'sniewski's typed, free logic with modal quantifiers, which proves a predicative version of $\HP$, can interpret $\PAT$.   Our setting is quite different from that of Urbaniak's paper as Le\'sniewski's typed, free logic differs dramatically from the one we work in here.   The work in this paper proceeds  by looking at whether a version of Frege's Theorem  is available in a classical second-order modal setting.    Unlike these other results, we find that a modal version of Frege's Theorem for $\PAT$ is \emph{not} possible, as shown by Theorem~\ref{thm:main2}.

%%%%%%%%%%%%%%%%%%%%%%%%%%%%%%%%
%%%%%%%%%%%%%%%%%%%%%%%%%%%%%%%%
%%%%%%%%%%%%%%%%%%%%%%%%%%%%%%%%
\subsection{Outline of paper}
%%%%%%%%%%%%%%%%%%%%%%%%%%%%%%%%
%%%%%%%%%%%%%%%%%%%%%%%%%%%%%%%%
%%%%%%%%%%%%%%%%%%%%%%%%%%%%%%%%

This paper is organised as follows. Section~\ref{sec:lang} expands the potentially infinite models' language to include the language of arithmetic. In Section~\ref{sec:q} we show that using the expanded language the potentially infinite models satisfy a weak theory of arithmetic equivalent to a modal version of Robinson's $\Q$. In Section~\ref{sec:iterp} we define the inductive formulas of the language and show that induction holds for them. This allows us to show Theorem~\ref{thm:main}, that $\TAO$ is interpretable in our external theory and $\PAO$ is interpretable in our internal theory, in Section~\ref{sec:interpertation}.  In Section~\ref{sec:prmain2} we show that no natural interpretation of $\PAT$ is possible by proving Theorem~\ref{thm:main2}.

%%%%%%%%%%%%%%%%%%%%%%%%%%%%%%%%%%%%%%%%%%%%%%%%%%%%%%%%%%%%%%%%%%%%%%%%%%%%%%%%%%%%%%%%%%%%%%%%%%%%%%%%%%%%%%%%%%%%%%%%%%%%%%%%%%%%%%%%%%%%%%%
%%%%%%%%%%%%%%%%%%%%%%%%%%%%%%%%
%%%%%%%%%%%%%%%%%%%%%%%%%%%%%%%%
%%%%%%%%%%%%%%%%%%%%%%%%%%%%%%%%
\section{Definitions for a Modal Grundlagen}\label{sec:lang}
%%%%%%%%%%%%%%%%%%%%%%%%%%%%%%%%
%%%%%%%%%%%%%%%%%%%%%%%%%%%%%%%%
%%%%%%%%%%%%%%%%%%%%%%%%%%%%%%%%

Just as Frege in the \emph{Grundlagen} defined the numbers and the relations on them using only the `number of' operator, here we show how modified versions of Frege's definitions can do this in the setting of the potentially infinite.\footnote{This has some precedent in Hodes (\citeyear[383]{Hodes1990}).  However, whereas we (and Frege) first define successor and then use this to build the other definitions, Hodes takes `less than or equal to' as his primitive.  In his system a number $N$ (understood as a higher-order object) is less than or equal to another number $N'$ just in case it is possible that there are two other second-order objects $A$ and $A'$ each with the same number of objects as $N$ and $N'$ respectively and $A$ is a subset of $A'$.  That this has parallels with the definition of successor offered here will be clear on inspection.} 
Proving that these definitions satisfy the usual arithmetical axioms will occupy us in \S\S\ref{sec:q}--\ref{sec:iterp}. In this section we simply set out the definitions themselves and say a word about their motivation. While entirely rigorous, it is our hope that, as in the \emph{Grundlagen}, the definitions will be intuitive and correspond to our understanding of cardinal numbers.

The first definition is easy and does not require any of the modal apparatus. We simply let $0=\#\varnothing$. This follows Frege~(\citeyear[\S74~p.~87]{Frege1950}) explicitly, who said that zero is ``the Number which belongs to the concept `not identical with itself'".  

%FIGURE
%%%%%%%%%%%%%%%%%%%%%%%%%%%%%%%%
\begin{figure}[t]
%SUBFIGURE 1
\begin{subfigure}[b]{0.3\textwidth}
%RESIZE SUBFIGURE 1
\scalebox{0.8}{
%TIKZ SUBFIGURE 1
\begin{tikzpicture}
%base
\draw (1,1) circle [radius=1];;
\draw (4,1) circle [radius=1];;
\draw [->] (2,1) -- (3,1);
\node at (2.4,1.25) {$\exists$};
%left points
\node [above] at (1,1.5) {a};
\draw[fill] (1,1.5) circle [radius=0.025];
\node [above] at (0.5,0.5) {b};
\draw[fill] (0.5,0.5) circle [radius=0.025];
%right set and point
%X
\draw (4,1) circle [radius=0.5];;
\node [above] at (4,1.5) {X};
%x
\node [above] at (4,1) {x};
\draw[fill] (4,1) circle [radius=0.025];
%text
\node [above] at (1,2) {$Sab$};
\node [above] at (4,2.5) {$a=\#X/\{x\}$};
\node [above] at (4,2) {$b=\#X$};
\end{tikzpicture}}
%CAPTION SUBFIGURE 1
\caption{Diagram of when a is succeeded by b.\\ \phantom{a}}
%LABEL SUBFIGURE 1
\label{fig:suc}
\end{subfigure}
%SUBFIGURE 2
\begin{subfigure}[b]{0.3\textwidth}
%RESIZE SUBFIGURE 2
\scalebox{0.8}{
%TIKZ SUBFIGURE 2
\begin{tikzpicture}
%base
\draw (1,1) circle [radius=1];;
\draw (4,1) circle [radius=1];;
\draw [->] (2,1) -- (3,1);
\node at (2.4,1.25) {$\exists$};
%left points
\node [above] at (1,1.5) {a};
\draw[fill] (1,1.5) circle [radius=0.025];
\node [above] at (0.5,0.5) {b};
\draw[fill] (0.5,0.5) circle [radius=0.025];
%right set and point
%X
\draw (3.75,1.25) circle [radius=0.5];;
\node [above] at (4.2,1.5) {X};
%x
\node [above] at (4.5,0.5) {x};
\draw[fill] (4.5,0.5) circle [radius=0.025];
%text
\node [above] at (1,2) {$Sab$};
\node [above] at (4,2.5) {$a=\#X$};
\node [above] at (4,2) {$b=\#X\cup\{x\}$};
\end{tikzpicture}}
%CAPTION SUBFIGURE 2
\caption{Diagram of when a is succeeded by b for alternative definition of S.}
%LABEL SUBFIGURE 2
\label{fig:sucalt}
\end{subfigure}
%SUBFIGURE 3
\begin{subfigure}[b]{0.3\textwidth}
%RESIZE SUBFIGUE 3
\scalebox{0.8}{
%TIKZ SUBFIGURE 3
\begin{tikzpicture}
%base
\draw (1,1) circle [radius=1];;
\draw (4,1) circle [radius=1];;
\draw [->] (2,1) -- (3,1);
\node at (2.4,1.25) {$\exists$};
%left points
\node [above] at (1,1.5) {a};
\draw[fill] (1,1.5) circle [radius=0.025];
\node [above] at (0.5,0.5) {b};
\draw[fill] (0.5,0.5) circle [radius=0.025];
\node [above] at (1.5,0.5) {c};
\draw[fill] (1.5,0.5) circle [radius=0.025];
%right set and point
%X
\draw (3.75,1.25) circle [radius=0.25];;
\node [above] at (3.75,1.5) {X};
%x
\node [above] at (4.25,0.75) {Y};
\draw[fill] (4.25,0.5) circle [radius=0.25];
%text
\node [above] at (1,2) {$+(a,b,c)$};
\node [above] at (4,3) {$a=\#X$};
\node [above] at (4,2.5) {$b=\#Y$};
\node [above] at (4,2) {$c=\#X\cup Y$};
\end{tikzpicture}}
%CAPTION SUBFIGURE 3
\caption{Diagram of when c is the addition of b and c.\\ \phantom{a}}
%LABEL SUBFIGURE 3
\label{fig:addition}
\end{subfigure}
%CAPTION FIGURE
\caption{}
%LABEL FIGURE
\label{fig:diagrams}
\end{figure}

Next we must define the successor, as the other definitions rely on it.  The definition here is like the one offered by Frege, but it differs by allowing the sets which witness that one object is the successor of another to be merely possible.  This is to ensure that if an object is ever the successor of another, then it is the successor of that object in every world where they both exist. This property will be important in the proof of induction. The definition of successor, in plain terms, is: one object is the successor of another just in case it is possible that there are two sets, which differ by one object and the successor is the number of the larger set, and the predecessor is the number of the smaller set. Figures~\ref{fig:suc} and \ref{fig:sucalt} illustrate the two ways this can be done, resulting in two definitions of the successor:

%DEFINITION
%%%%%%%%%%%%%%%%%%%%%%%%%%%%%%%%
\begin{definition}
%EQUATION
%%%%%%%%%%%%%%%%%%%%%%%%%%%%%%%%
\begin{equation}\label{suc}
Sxy\equiv\Diamond\exists G, u[Gu\wedge (y=\#G)\wedge(x=\#(G-\set{u}))]
\end{equation}
%EQUATION
%%%%%%%%%%%%%%%%%%%%%%%%%%%%%%%%
\begin{equation}\label{suc2}
S'xy\equiv\Diamond\exists F, u[\neg Fu\wedge (x=\#F)\wedge(y=\#(F\cup\set{u}))]
\end{equation}
%%%%%%%%%%%%%%%%%%%%%%%%%%%%%%%%
\end{definition}
%%%%%%%%%%%%%%%%%%%%%%%%%%%%%%%%

\noindent The first of these definitions simply adds the possibility operator to the definition of successor suggested by  Frege~(\citeyear[\S76~p.~89]{Frege1950}).  These definitions are equivalent: to see this, simply consider $F=G-\set{u}$ and $G=F\cup\set{u}$.\footnote{For easy of readability, we will use set theoretic notation as a convenient short hand for concepts formed using the language of the model.  So $F\cup\set{u}$ is used for the concept given by $Xx\leftrightarrow (Fx\vee x=u)$.}  In what follows we will simply use the definition that is most convenient and will write $S$ for both.

The definition of addition is similarly intuitive. The relation $+$ holds between three objects $a$, $b$, and $c$ such that it is possible that there are disjoint sets $X$ and $Y$ of cardinality $a$ and $b$ respectively, and $c$ is the cardinality of $X\cup Y$, the union of the two disjoint sets. This is illustrated by Figure~\ref{fig:addition} and can be written formally as:

%DEFINITION
%%%%%%%%%%%%%%%%%%%%%%%%%%%%%%%%
\begin{definition}\label{addit}
%EQUATION
%%%%%%%%%%%%%%%%%%%%%%%%%%%%%%%%
\begin{equation}
    +\!(a,b,c)\equiv\Diamond\exists X,Y(a=\#X\wedge b=\#Y\wedge c=\#X\cup Y\wedge (X\cap Y)=\varnothing)
\end{equation}
%%%%%%%%%%%%%%%%%%%%%%%%%%%%%%%%
\end{definition}
%%%%%%%%%%%%%%%%%%%%%%%%%%%%%%%%

For $c$ to be the result of multiplying $a$ and $b$ we need a set $B$ of cardinality $b$ and for each element $x$ of $B$ a set $A_x$ of cardinality $a$.  The $A_x$'s must all be disjoint.  And $c$ must be the cardinality of the union of all the $A_x$'s. To define the $A_x$'s we define a binary relation $P$ that holds between $x$ in $B$ and all $y$ in $A_x$.  So $A_x$ is $\set{y\mid Pxy}$.

%DEFINITION
%%%%%%%%%%%%%%%%%%%%%%%%%%%%%%%%
\begin{definition}\label{multip}
%EQUATION
%%%%%%%%%%%%%%%%%%%%%%%%%%%%%%%%
\begin{multline}
    \times\!(a,b,c)\equiv\Diamond\exists X,P[\#X=b\wedge \forall x\in X(\#\set{y\mid Pxy}=a)\\\wedge \forall x,y\in X(x\neq y\rightarrow\set{z\mid Pxz}\cap\set{z\mid Pyz}=\varnothing)\wedge \#\bigcup_{x\in X}\set{y\mid Pxy}=c]
\end{multline}
%%%%%%%%%%%%%%%%%%%%%%%%%%%%%%%%
\end{definition}
%%%%%%%%%%%%%%%%%%%%%%%%%%%%%%%%

The definition of the natural numbers is more complicated and require us to define the notion that one number follows another in the ordering of the natural numbers. We will make use of Frege's definition from the 1879 \emph{Begriffsschrift} (\citeyear[\S III pp.~55~ff]{Heijenoort1967aa}; \citeyear[\S79~p.~92~ff]{Frege1950}).     Russell and Whitehead (\citeyear[316]{whitehead1910}) called this relation the \textit{ancestral relation} because a good example of what it does is define the relation `ancestor of' from  the relation `parent of'.  The \textit{strong ancestral} of $\varphi$ holds between two objects $a$ and $b$ just in case $b$ is contained in every set such that the set is closed under $\varphi$ and the set contains everything $a$ bears $\varphi$ to.  So, we can define someone's ancestors as everyone who is in every set that contains their parents and the parents of everyone in the set.  It is not guaranteed that $a$ bears this relation to itself, and so we also define the reflexive \textit{weak ancestral}.

%DEFINITION
%%%%%%%%%%%%%%%%%%%%%%%%%%%%%%%%
\begin{definition}[The strong ancestral]\label{sances}
%EQUATION
%%%%%%%%%%%%%%%%%%%%%%%%%%%%%%%%
\[ \varphi^+(a,b)\equiv\forall X[(\forall x,y(Xx\wedge \varphi (x,y)\rightarrow Xy)\wedge\forall x(\varphi (a,x)\rightarrow Xx))\rightarrow Xb]. \]
%%%%%%%%%%%%%%%%%%%%%%%%%%%%%%%%
\end{definition}
%%%%%%%%%%%%%%%%%%%%%%%%%%%%%%%%

%DEFINITION
%%%%%%%%%%%%%%%%%%%%%%%%%%%%%%%%
\begin{definition}[The weak ancestral]\label{wances}
%EQUATION
%%%%%%%%%%%%%%%%%%%%%%%%%%%%%%%%
\[\varphi ^{+=}(a,b)\equiv \varphi ^+(a,b)\vee a=b.\]
%%%%%%%%%%%%%%%%%%%%%%%%%%%%%%%%
\end{definition}
%%%%%%%%%%%%%%%%%%%%%%%%%%%%%%%%

\noindent Using this definition, we define a natural number as an object that is some finite number of successor steps from 0, assuming 0 exists.

%DEFINITION
%%%%%%%%%%%%%%%%%%%%%%%%%%%%%%%%
\begin{definition}[Natural Number]\label{nn}
%EQUATION
%%%%%%%%%%%%%%%%%%%%%%%%%%%%%%%%
$$\mathbb{N}x\equiv S^{+=}0x\wedge\exists y(y=0).$$
%%%%%%%%%%%%%%%%%%%%%%%%%%%%%%%%
\end{definition}
%%%%%%%%%%%%%%%%%%%%%%%%%%%%%%%%

\noindent This definition closely parallels Frege's, though the definition of $S$ is different.  The existence claim is added because in the modal setting 0's existence cannot be assumed.  For example, 0 does not exist at worlds \textbf{0}, \textbf{1}, and \textbf{2} in the \textbf{0}-\textbf{3} swap model, and, as \textbf{0} is not a member of infinitely many finite subsets of the natural numbers, 0 does not exist at infinitely many worlds in the subset model. In these worlds nothing is a natural number.

%%%%%%%%%%%%%%%%%%%%%%%%%%%%%%%%
%%%%%%%%%%%%%%%%%%%%%%%%%%%%%%%%
%%%%%%%%%%%%%%%%%%%%%%%%%%%%%%%%
\subsection{Some useful results}\label{sec:usefulresults}
%%%%%%%%%%%%%%%%%%%%%%%%%%%%%%%%
%%%%%%%%%%%%%%%%%%%%%%%%%%%%%%%%
%%%%%%%%%%%%%%%%%%%%%%%%%%%%%%%%

The following six lemmas will help explain the behaviour of $\nn$ in the models. We admit the proofs as they do not pose any particular difficulty. For the following Lemmas, recall Definition~\ref{pivalid} where $\vDash_{\mathsf{PI}}\varphi$ was defined as $\varphi$ is true in all worlds in all potentially infinite models.  First, note that the set defined by $\nn$ at a world satisfies the antecedent of $S^+0x$.  Intuitively, the idea here is that if $x$ is in every set containing 0 and closed under $S$, and $Sxy$, or $S0y$, then $y$ must also be in every set with these properties. 

%LEMMA
%%%%%%%%%%%%%%%%%%%%%%%%%%%%%%%%
\begin{lem}\label{lem:nnante2}
$\vDash_{\mathsf{PI}}\exists x(x=0)\rightarrow\forall y(S0y\rightarrow\nn y))$
\end{lem}
%%%%%%%%%%%%%%%%%%%%%%%%%%%%%%%%

%LEMMA
%%%%%%%%%%%%%%%%%%%%%%%%%%%%%%%%
\begin{lem}\label{lem:nnante1}
$\vDash_{\mathsf{PI}}\forall x,y(\nn x\wedge Sxy\rightarrow\nn y)$
\end{lem}
%%%%%%%%%%%%%%%%%%%%%%%%%%%%%%%%

\noindent It follows immediately from this that if $x$ exists at a world and at that world $\nn y$ and $Syx$ then $\nn x$.  However, that doesn't mean $\nn$ is the set of all numbers across all worlds as $\nn$ only holds of objects which exist at the world of evaluation.  This contrasts with our other definitions where the objects need not exist at the world.  

%LEMMA
%%%%%%%%%%%%%%%%%%%%%%%%%%%%%%%%
\begin{lem}\label{lem:noexist}
$\models_{PI} \mathbb{N}x \rightarrow\exists y \; y=x$
\end{lem}
%%%%%%%%%%%%%%%%%%%%%%%%%%%%%%%%

\noindent This is because the quantifiers in $\nn$ are plain rather than having modals in front of them.  This is important because if we put the modals in front everything is a number!

 We informally extend our definition of the interpretation function $I$ to $I(\nn,s)=\set{x\in D(s)\mid \momodels{s}{g}\nn x}$. Note that by Lemma~\ref{lem:noexist} we have $\set{x\in D(s)\mid \momodels{s}{g}\nn x}=\set{x\in D\mid \momodels{s}{g}\nn x}$, where $D$ is the domain of the model not the world.  

Recall that $\mathbf{a}_i$ is the unique element in $D$ such that if $|X|=i$ then $I(\#,w)(X)=\mathbf{a}_i$ as defined in \ref{jo}.5. We can now explicitly describe the interpretation of $\nn$ at a world $w$ in terms of the $\mathbf{a}_i$'s, that is, the set $I(\nn,w)$:

%LEMMA
%%%%%%%%%%%%%%%%%%%%%%%%%%%%%%%%
\begin{lem}\label{fact:nnatworld}
    Let $w$ be a world and let $n$ be the first number such that $\mathbf{a}_n\notin D(w)$. Then if $n>0$, it follows that $\set{0,\mathbf{a}_1,\dots,\mathbf{a}_{n-1}}=I(\nn,w)$, and further, $n=0$ iff $I(\nn,w)=\varnothing$.
\end{lem}
%%%%%%%%%%%%%%%%%%%%%%%%%%%%%%%%

\noindent This result shows us how the differences between our modal setting and the traditional non-modal setting of the \emph{Grundlagen} become most stark in the case of the interpretation of the natural numbers at a world.  Two things are worth highlighting.  The first is that $\nn$ is finite at every world, since it is a subset of the domain of the world, and the domain of every world is finite.  The second is that objects that are not in $\nn$ at one world can `become' numbers at later worlds.   This doesn't happen in the minimal model, where $I(\nn,\mathbf{n})=D(\mathbf{n})$ at every world.  But it does in the subset model.  For example, $I(\nn,\set{\mathbf{2},\mathbf{100}})=\varnothing$, $I(\nn,\set{\mathbf{0},\mathbf{1},\mathbf{3}})=\set{\mathbf{0},\mathbf{1}}$ and $I(\nn,\set{\mathbf{0},\mathbf{1},\mathbf{2},\mathbf{3},\mathbf{100}})=\set{\mathbf{0},\mathbf{1},\mathbf{2},\mathbf{3}}$.   This distinguishes $\neg\nn(x)$ from the other relations which have a certain stability; if objects stand in these relations at one world, then they do so in all worlds in which they all exist. The formal definition of stability is given as Definition~\ref{stable}.  This difference is caused by there being no possibility operator at the beginning of the definition of $\nn$.  Despite this, once something is a number it remains one: 

%LEMMA
%%%%%%%%%%%%%%%%%%%%%%%%%%%%%%%%
\begin{lem}\label{lem:sstability}
    $\vDash_{\mathsf{PI}}S(x,y)\rightarrow\Box S(x,y)$ holds, as does $\vDash_{\mathsf{PI}}S^+(x,y)\rightarrow\Box S^+(x,y)$, $\vDash_{\mathsf{PI}}S^{+=}(x,y)\rightarrow\Box S^{+=}(x,y)$ and $\vDash_{\mathsf{PI}}\nn x\rightarrow\Box \nn x$.
\end{lem}
%%%%%%%%%%%%%%%%%%%%%%%%%%%%%%%%

\noindent It is also worth noting that even though some cardinalities may not be numbers at ever world, the cardinality of every set eventually becomes a natural number.

%LEMMA
%%%%%%%%%%%%%%%%%%%%%%%%%%%%%%%%
\begin{lem}\label{lem:everynumbersomewhere}
    For all $w\in W$ and  $X\subseteq D(w)$, there is a world $s$ such that $R(w,s)$ and $\#X\in I(\nn,s)$.
\end{lem}
%%%%%%%%%%%%%%%%%%%%%%%%%%%%%%%%

\noindent This is because $\#$ is a function, first-order converse Barcan holds, and the accessibility relation is directed.  With these preliminary results we can now show our definitions satisfy a simple theory of arithmetic. 

%%%%%%%%%%%%%%%%%%%%%%%%%%%%%%%%%%%%%%%%%%%%%%%%%%%%%%%%%%%%%%%%%%%%%%%%%%%%%%%%%%%%%%%%%%%%%%%%%%%%%%%%%%%%%%%%%%%%%%%%%%%%%%%%%%%%%%%%%%%%%%%%%%%%%%%%%%%%%%%%%%
%%%%%%%%%%%%%%%%%%%%%%%%%%%%%%%%
%%%%%%%%%%%%%%%%%%%%%%%%%%%%%%%%
%%%%%%%%%%%%%%%%%%%%%%%%%%%%%%%%
\section{Proving Modalized Robinson's $\Q$}\label{sec:q}
%%%%%%%%%%%%%%%%%%%%%%%%%%%%%%%%
%%%%%%%%%%%%%%%%%%%%%%%%%%%%%%%%
%%%%%%%%%%%%%%%%%%%%%%%%%%%%%%%%

In what follows we will prove that the modalized axioms of Robinson's $ \Q$ are true on all PI models (cf. Definition~\ref{pivalid}).  Robinson's $\Q$ is a weak theory of arithmetic that defines successor as an injective function that never returns 0 and gives a recursive definition of addition and multiplication. By ``modalized" we mean that we write ``necessarily for all" for ``for all" and ``possibly there is'' for ``there is''.  In other words, it is what results when we apply the Linnebo translation (mentioned in the introduction) to the axioms of Robinson's $\Q$.   The theory $\PAO$ is obtained by adding the mathematical induction schema to $\Q$.  We deal with $\PAO$ and the proof of the induction schema in Section~\ref{sec:iterp}.\footnote{A list of the non-modalized axioms can be found in Appendix~\ref{ap:theories}.  While what we show here is that these axioms are in the theory $\EPI$, each of the proofs that follow can be formalised in $\ACA$  (cf. Appendix~\ref{ap:piintoapa0}).  That this is possible will ensures that all axioms proven here are also in the theory $\PI$ (from Section~\ref{sec:mainresults}). This is a key point in the proof of Theorem~\ref{thm:main}.ii which we complete in section~\ref{sec:interpertation}.}

First we will show that our relations define the graphs of functions.  The easiest case is successor.

%LEMMA
%%%%%%%%%%%%%%%%%%%%%%%%%%%%%%%%
\begin{lem}[S1]\label{sucfunc}
$\vDash_{\mathsf{PI}}\Box\forall x,y,z\in\nn (( S xy\wedge S xz)\rightarrow y=z)$.
\end{lem}
%%%%%%%%%%%%%%%%%%%%%%%%%%%%%%%%

%PROOF
%%%%%%%%%%%%%%%%%%%%%%%%%%%%%%%%
\begin{proof}
Let $s\in W$ and $x,y,z\in I(\nn,s)$ satisfy the antecedent.  As $x$ is the predecessor in both relations it follows by directedness that there is a $w\in W$,  such that $R(s,w)$ where there are  $X,X'\subseteq D(w)$ and $\#X=x=\#X'$. As such there is a bijection $g:X\rightarrow X'$.  There will also be $a,b\in D(w)$ such that $a\notin X$, $b\notin X'$, and $y=\#X\cup\set{a}$ and $z=\#X'\cup\set{b}$. As $a\notin X$ and $b\notin X'$ we can construct $h$ such that for all $u\in X$, $h(u)=g(u)$ and $h(a)=b$.  Clearly $h$ is a bijection, so $y=\#X\cup\set{a}=\#X'\cup\set{b}=z$.
\end{proof}
%%%%%%%%%%%%%%%%%%%%%%%%%%%%%%%%

%LEMMA
%%%%%%%%%%%%%%%%%%%%%%%%%%%%%%%%
\begin{lem}[S2]\label{phoibos}
$\vDash_{\mathsf{PI}}\Box\forall x\in\nn\Diamond\exists y\in\nn \; S xy$.
\end{lem}
%%%%%%%%%%%%%%%%%%%%%%%%%%%%%%%%

%PROOF
%%%%%%%%%%%%%%%%%%%%%%%%%%%%%%%%
\begin{proof}  As illustrated in Figure~\ref{fig:existsuccessor}, let $s\in W$ and $x\in I(\nn,s)$, it follows that $x=\mathbf{a}_n$ for some $n$ and, by Lemma~\ref{fact:nnatworld}, $\set{0,\dots\mathbf{a}_{n-1}}\subsetneq D(s)$.  Further, $\mathbf{a}_n=\#\set{0,\dots\mathbf{a}_{n-1}}$ and $\mathbf{a}_n\notin\set{0,\dots\mathbf{a}_{n-1}}$.  Thus, there must be a further world $w$ accessible from $w_1$ and a $y\in D(w)$ such that $y=\#\set{0,\dots\mathbf{a}_{n-1}}\cup\set{\mathbf{a}_n}$.  It follows that $Sxy$ at $w$.  By Lemma~\ref{lem:sstability} $x\in I(\nn,w)$. As $\nn$ is closed under successor by Lemma~\ref{lem:nnante1}, we have that $ y\in I(\nn,w)$. And since $R$ is transitive, $w$ is accessible from $s$. 
\end{proof}
%%%%%%%%%%%%%%%%%%%%%%%%%%%%%%%%

%FIGURE
%%%%%%%%%%%%%%%%%%%%%%%%%%%%%%%%
\begin{figure}
%CENTERING FIGURE
\centering
%SUBFIGURE 1
\begin{subfigure}{.49\textwidth}
    %CENTERING SUBFIGURE 1
    \centering
    %RESIZING SUBFIGURE 1
    \resizebox{.8\textwidth}{!}{
    %TIKZ SUBFIGURE 1
    \begin{tikzpicture}
%base
\draw (1,1) circle [radius=1];;
\draw (4,1) circle [radius=1];;
\draw [->] (2,1) -- (3,1);
\node at (2.4,1.25) {$\exists$};
\draw (7,1) circle [radius=1];;
\draw [->] (5,1) -- (6,1);
\node at (5.4,1.25) {$\exists$};
\draw (10,1) circle [radius=1];;
\draw [->] (8,1) -- (9,1);
\node at (8.4,1.25) {$\exists$};
\draw [->] (1,2) to [out=30,in=150] (10,2);
%points
\node [above] at (1,1) {x};
\draw[fill] (1,1) circle [radius=0.025];
\node [above] at (10,1) {y};
\draw[fill] (10,1) circle [radius=0.025];
\node [above] at (7.5,0.5) {a};
\draw[fill] (7.5,0.5) circle [radius=0.025];
%X
\draw (3.75,1.25) circle [radius=0.25];;
\node [above] at (3.75,1.5) {X};
\draw (6.75,1.25) circle [radius=0.25];;
\node [above] at (6.75,1.5) {X};
%text
\node [below] at (4,0) {$x=\#X$};
\node [below] at (10,0) {$y=\#X\cup\{a\}$};
\end{tikzpicture}}
    %CAPTION SUBFIGURE 1
    \caption{Finding the successor of $x$}
    %LABEL SUBFIGURE 1
    \label{fig:existsuccessor}
\end{subfigure}
    %SUBFIGURE 2
    \begin{subfigure}{.49\textwidth}
    %CENTERING SUBFIGURE 2
    \centering
    %RESIZING SUBFIGURE 2
    \resizebox{.8\textwidth}{!}{
    %TIKZ SUBFIGURE 2
    \begin{tikzpicture}
    %base
\draw[thick,dashed] (1,1) circle [radius=1];;
\draw (1,.7) circle [radius=.5];;
\node [above] at (1,.4) {$Y_0$};
\draw (4,.7) circle [radius=.7];;
\draw [<->] (1.5,.7) -- (3.3,.7);
\node at (2.4,.9) {$g_0$};
\draw [<->] (1,1.6) -- (4,1.6);
\node at (2.4,1.8) {$o$};
\draw (7,1) circle [radius=1];;
\draw [<->] (4.65,1.2) -- (6,1.2);
\node at (5.4,1.4) {$f_0$};
\draw [<->] (4.6,.4) -- (6.2,.4);
\node at (5.4,.6) {$f_1$};
%line
\draw (3.3,.7) -- (4.7,.7);
\draw[thick,dashed] (3.3,0.7) arc(190:-10:0.7cm and 1.25cm);
\draw (6,1) -- (8,1);
%points
\node [above] at (1,1.6) {x};
\draw[fill] (1,1.6) circle [radius=0.025];
\node [above] at (4,1.6) {y};
\draw[fill] (4,1.6) circle [radius=0.025];
%X
\node [above] at (3.75,.7) {$X_0$};
\node [above] at (3.75,.2) {$N$};
\node [above] at (6.75,1.2) {$X_1$};
\node [above] at (6.75,.2) {$N$};
\end{tikzpicture}}
    %CAPTION SUBFIGURE 2
    \caption{Proof of the recursion clause for addition}
    %LABEL SUBFIGURE 2
    \label{fig:existadd}
\end{subfigure}
%CAPTION FIGURE
\caption{}
\end{figure}
%%%%%%%%%%%%%%%%%%%%%%%%%%%%%%%%

\noindent These two proofs offer a general outline of the reasoning for addition and multiplication.  For S1 this strategy is to show that whatever $x$ is the sets assigned to $y$ and $z$ will have the same cardinality.  Where as for S2 one simply needs to construct a set of the correct cardinality. For this reason we do not give the proofs for the next four lemmas.

%LEMMA
%%%%%%%%%%%%%%%%%%%%%%%%%%%%%%%%
\begin{lem}[A1]\label{addfunc}
    $\vDash_{\mathsf{PI}}\Box\forall x,y,z,z'\in\nn(+(x,y,z)\wedge +(x,y,z')\rightarrow z=z')$.
\end{lem}
%%%%%%%%%%%%%%%%%%%%%%%%%%%%%%%%

%LEMMA
%%%%%%%%%%%%%%%%%%%%%%%%%%%%%%%%
\begin{lem}[A2]\label{vergial}
    $\vDash_{\mathsf{PI}}\Box\forall x,y\in \nn\Diamond\exists z\in\nn \ +(x,y,z)$.
\end{lem}
%%%%%%%%%%%%%%%%%%%%%%%%%%%%%%%%
 
%LEMMA
%%%%%%%%%%%%%%%%%%%%%%%%%%%%%%%%
\begin{lem}[M1]\label{multfunc1}
    $\vDash_{\mathsf{PI}}\Box\forall x,y,z,z'\in\nn(\times(x,y,z)\wedge \times(x,y,z')\rightarrow z=z')$.
\end{lem}
%%%%%%%%%%%%%%%%%%%%%%%%%%%%%%%%

%LEMMA
%%%%%%%%%%%%%%%%%%%%%%%%%%%%%%%%
\begin{lem}[M2]\label{multfunc2}
    $\vDash_{\mathsf{PI}}\Box\forall x,y\in \nn\Diamond\exists z\in\nn \ \times(x,y,z)$.
\end{lem}
%%%%%%%%%%%%%%%%%%%%%%%%%%%%%%%%

We also need to show that $0$ meets the right conditions to be a constant.

%LEMMA
%%%%%%%%%%%%%%%%%%%%%%%%%%%%%%%%
\begin{lem}[Z1]\label{Korn}
$\vDash_{\mathsf{PI}}\Diamond\exists x\in\nn (x=0\wedge \Box\forall y(y=0\rightarrow y=x))$.
\end{lem}
%%%%%%%%%%%%%%%%%%%%%%%%%%%%%%%%

%PROOF
%%%%%%%%%%%%%%%%%%%%%%%%%%%%%%%%
\begin{proof}
By the definition of $\nn$, it follows that $0\in I(\nn, s)$ for any world $s$ with 0 in the domain.  And as $0=\#\varnothing$ there is some $s$ with 0 in the domain.  The second conjunct follows by the transitivity of identity.  
\end{proof}
%%%%%%%%%%%%%%%%%%%%%%%%%%%%%%%%

\noindent We can now move on to the recursion equations in $\Q$.  We separate these into the base steps concerning 0 and the recursive step.  For the base steps, because $0=\#\varnothing$ the proofs of the lemmas are relatively straight forward.  As such we list them here without proof.  

%LEMMA
%%%%%%%%%%%%%%%%%%%%%%%%%%%%%%%%
\begin{lem}[Q1]\label{sucrec2}
$\vDash_{\mathsf{PI}}\neg\Diamond\exists x\in \nn (S x0 )$.
\end{lem}
%%%%%%%%%%%%%%%%%%%%%%%%%%%%%%%%

%LEMMA
%%%%%%%%%%%%%%%%%%%%%%%%%%%%%%%%
\begin{lem}[Q3]\label{addrec1}
    $\vDash_{\mathsf{PI}}\Box\forall x\in\nn \ +(x,0,x)$.
\end{lem}
%%%%%%%%%%%%%%%%%%%%%%%%%%%%%%%%

%LEMMA
%%%%%%%%%%%%%%%%%%%%%%%%%%%%%%%%
\begin{lem}[Q5]\label{multrec1}
    $\vDash_{\mathsf{PI}}\Box\forall x\in\nn \ \times(x,0,0)$.
\end{lem}
%%%%%%%%%%%%%%%%%%%%%%%%%%%%%%%%

\noindent What is left now is to show the recursion steps.  He we only prove the case for $+$ as one can use the same stratagy for $\times$ and the proof is simple for $S$.

%LEMMA
%%%%%%%%%%%%%%%%%%%%%%%%%%%%%%%%
\begin{lem}[Q2]\label{sucrec1}
$\vDash_{\mathsf{PI}}\Box\forall x,y,z\in\nn ((S xz\wedge S yz)\rightarrow x=y)$.
\end{lem}
%%%%%%%%%%%%%%%%%%%%%%%%%%%%%%%%

\noindent The proof simply follows from the fact that if there is a bijection between two sets $X$ and $Y$ then there will be a bijection between $X\cup\set{a}$ and $Y\cup\set{b}$ if $a$ and $b$ aren't in $X$ or $Y$ respectively.

%LEMMA
%%%%%%%%%%%%%%%%%%%%%%%%%%%%%%%%
\begin{lem}[Q4]\label{add4}
    $$\vDash_{\mathsf{PI}}\Box\forall n,x_0,x_1,y_0,y_1,z\in\nn(S(x_0,x_1)\wedge S(y_0,y_1)\wedge +(n,x_0,y_0)\wedge+(n,x_1,z)\rightarrow y_1=z).$$
\end{lem}
%%%%%%%%%%%%%%%%%%%%%%%%%%%%%%%%

%PROOF
%%%%%%%%%%%%%%%%%%%%%%%%%%%%%%%%
\begin{proof}
 As illustrated in Figure~\ref{fig:existadd}, let $s\in W$ and $n,x_0,x_1,y_0,y_1,z\in I(\nn,s)$ satisfy the antecedent.  We want to show that $y_1=z$.  By directedness, we know there is a world $w$ containing all the objects and sets which the antecedent states possibly exist.  As $y_1$ succeeds $y_0$ there is a set $Y_0$ and an object $a\notin Y_0$ at $w$ such that $y_1=\#Y_0\cup\set{a}$ and $y_0=\#Y_0$.  We know $y_0$ to be the addition of $n$ and $x_0$ so there are disjoint sets $N$ and $X_0$ such that $n=\#N$, $x_0=\#X_0$, and $y_0=\#N\cup X_0$.  Further there is a bijection $g_0:Y_0\rightarrow N\cup X_0$.  Now let $b$ be an element not in $N$ or $X_0$ (we can always pick $w$ so that such an element exists).  Clearly we can define a bijection $o$ between the singletons of $a$ and $b$.  Now, using $g_0$ and $o$, define the bijection $g:Y_0\cup\set{a}\rightarrow N\cup X_0\cup\set{b}$, as the union of $g_0$ and $o$. Now as $x_1$ is the successor of $x_0$, it follows that $x_1=\#X_0\cup\set{b}$.  As $z$ is the addition of $n$ and $x_1$ there are disjoint sets $N'$ and $X_1$ such that $n=\#N=\#N'$, $x_1=\#X_0\cup\set{b}=\#X_1$ and $z=\#N\cup X_1$.  As such there are bijections $f_0:X_0\cup\set{b}\rightarrow X_1$ and $f_1:N\rightarrow N'$.  So, we can define the bijection $f:N\cup X_0\cup\set{b}\rightarrow N'\cup X_1$ as $f_0$ on $X_0\cup\set{b}$ and $f_1$ on $N$.  Then as $z=\#N\cup X_1$ the composition $f\circ g$ is a bijection proving $y_1=z$.
\end{proof}
%%%%%%%%%%%%%%%%%%%%%%%%%%%%%%%%

%LEMMA
%%%%%%%%%%%%%%%%%%%%%%%%%%%%%%%%
\begin{lem}[Q6]\label{stu}
    $\vDash_{\mathsf{PI}}\Box\forall n,x_0,x_1,y_0,y_1,z\in\nn(S(x_0,x_1)\wedge +(n,y_0,y_1)\wedge \times(n,x_0,y_0)\wedge\times(n,x_1,z)\rightarrow y_1=z)$.
\end{lem}
%%%%%%%%%%%%%%%%%%%%%%%%%%%%%%%%

\noindent This proof is similar to the above except we end up showing that $y_1=\#\bigcup_{x\in A_0\cup\set{u}}\set{y\mid Pxy\vee (x=u\wedge y\in N)}=\#\bigcup_{x\in A_1}\set{y\mid Txy}=z$ where $A_0$, $A_1$, and $N$ are of cardinality $x_0$, $x_1$, and $n$ respectively and $P$ is the relation given by $\times(n,x_0,y_0)$ and $T$ by $\times(n,x_1,z)$.

These results show that we have successfully defined a modalized version of Robinson's $\Q$ in our system.  The next section will recover a modalized induction schema.

%%%%%%%%%%%%%%%%%%%%%%%%%%%%%%%%%%%%%%%%%%%%%%%%%%%%%%%%%%%%%%%%%%%%%%%%%%%%%%%%%%%%%%%%%%%%%%%%%%%%%%%%%%%%%%%%%%%%%%%%%%%%%%%%%%%%%%%%%%%%%%%%%%%%%%%%%
%%%%%%%%%%%%%%%%%%%%%%%%%%%%%%%%
%%%%%%%%%%%%%%%%%%%%%%%%%%%%%%%%
%%%%%%%%%%%%%%%%%%%%%%%%%%%%%%%%
\section{Proving the Modalized Induction Schema}\label{sec:iterp}
%%%%%%%%%%%%%%%%%%%%%%%%%%%%%%%%
%%%%%%%%%%%%%%%%%%%%%%%%%%%%%%%%
%%%%%%%%%%%%%%%%%%%%%%%%%%%%%%%%

We have succeeded in giving a weak theory of arithmetic in a potentially infinite setting.  However, we can recover more arithmetic by proving that when restricted to appropriate formulas a modalized version of the induction schema is true on all PI models.  The modalized induction schema is:
%EQUATION
%%%%%%%%%%%%%%%%%%%%%%%%%%%%%%%%
\begin{equation}
    [\varphi(0)\wedge\Box\forall x,y\in\nn(\varphi(x)\wedge S(x,y)\rightarrow \varphi(y))]\rightarrow\Box\forall x\in\nn \ \varphi(x)
\end{equation}
%%%%%%%%%%%%%%%%%%%%%%%%%%%%%%%%

\noindent Modalized induction does not hold for all formulas in our models, as will be shown in Lemma~\ref{lem:nonstablenoinduct}. So, we need to define a subclass of the formulas in the language of potentially infinite models for which it does hold. These we will call the inductive formulas, and in Lemma~\ref{carl} it will be proven that induction does hold for inductive formulas.\footnote{This terminology is used to distinguish between these formulas and other for which induction does not hold.  Hopefully no confusion will be caused by the distinct uses of the term inductive formulas elsewhere in the literature.}

%DEFINITION
%%%%%%%%%%%%%%%%%%%%%%%%%%%%%%%%
\begin{definition}\label{def:inductiveformulas}
The \emph{inductive terms} and \emph{formulas} are defined recursively as follows:  
%LIST
%%%%%%%%%%%%%%%%%%%%%%%%%%%%%%%%
\begin{enumerate}
    \item An inductive term is either 0 or a first-order variable. 
    \item If $t_0,t_1,t_2$ are inductive terms then $t_0=t_1$, $S(t_0,t_1)$, $+(t_0,t_1,t_2)$ and $\times(t_0,t_1,t_2)$ are inductive formulas. 
    \item Applications of the propositional connectives to inductive formulas are inductive formulas.
    \item If $\varphi$ is an inductive formula then $\Box\forall x\in\nn \ \varphi$ and $\Diamond\exists x\in\nn \ \varphi$ are inductive formulas.
\end{enumerate}    
%%%%%%%%%%%%%%%%%%%%%%%%%%%%%%%%
\end{definition}
%%%%%%%%%%%%%%%%%%%%%%%%%%%%%%%%

\noindent The inductive terms and formulas are a subset of the terms and formulas respectively.  Any term of the form $\#X$ is not an inductive term, and indeed no term or formula with a free second-order variable is inductive.  Likewise $\nn0$, $\forall z (x=z)$ and $\exists y (S0y)$ are not inductive formulas,  while $\Box\forall z\in\nn(x=z)$ and $\Diamond\exists y\in\nn (S0y)$ are.

A formula $\varphi$  is \textit{stable} when: 
%EQUATION
%%%%%%%%%%%%%%%%%%%%%%%%%%%%%%%%
\begin{equation}\label{stable}
    \vDash_{\mathsf{PI}}\varphi\rightarrow\Box\varphi.
\end{equation}
%%%%%%%%%%%%%%%%%%%%%%%%%%%%%%%%

\noindent Stability is taken from Linnebo's (\citeyear[211]{Linnebo2013}) work on set theory in a modal setting.  It means once a formula has been made true it stays true.    As we saw in Lemma~\ref{lem:sstability}, $S$, $S^+,$ $S^{+=},$ and $\nn$ are all stable and an example of an unstable formula is $\neg\nn$. Fortunately, the inductive formulas all have the property of being stable, as we will now prove. This will allow us to prove induction for these formulas.

%LEMMA
%%%%%%%%%%%%%%%%%%%%%%%%%%%%%%%%
\begin{lem}\label{maire}
    If $\varphi$ is an inductive formula then $\vDash_{\mathsf{PI}}\varphi\rightarrow\Box\varphi$.
\end{lem}
%%%%%%%%%%%%%%%%%%%%%%%%%%%%%%%%

%PROOF
%%%%%%%%%%%%%%%%%%%%%%%%%%%%%%%%
\begin{proof} 
    In what follows we prove by induction on the complexity of the inductive formulas that both $\varphi\rightarrow\Box\varphi$ and $\Diamond\varphi\rightarrow\varphi$.  The second condition is included to deal with the case of negation. 

    Base case:      $x=y$ and $x=0$:  The result follows from the evaluation of $\#\varnothing$ being rigid and the identity relation being interpreted as the identity from the metalanguage.      Note that for $S, +,$ and $\times$ that $\Diamond\psi\rightarrow\psi$ follows simply because $R$ is transitive and they start with a $\Diamond$.  $S(x,y)$: See Lemma~\ref{lem:sstability}.    $+(x,y,z)$: Assume that $\momodels{w}{g}+(a,b,c)$.  It follows that there exists a world $w'$ accessible from $w$ and nonintersecting sets $A,B\subseteq{D(w')}$ satisfying $+$.  Let $s$ be a world such that $R(w,s)$. Then by directedness, there is a world $s'$ such that $R(s,s')$ and $R(w',s')$, and $A,B\subseteq{D(s')}$.  So $+(a,b,c)$ holds at $s$.     $\times(x,y,z)$:  The reasoning is essentially the same as that used for $+$.

    Now we proceed to the induction step. We will only show the case of the quantifier as $\neg$ and $\wedge$ proceed as one would expect. $\Diamond\exists x\in\nn \ \psi$: Assume $\momodels{s}{g}\Diamond\Diamond\exists x\in\nn \ \psi$. It follows by transitivity that $\momodels{s}{g}\Diamond\exists x\in\nn \ \psi$.  Now we show that $(\Diamond\exists x\in\nn \ \psi)\rightarrow(\Box\Diamond\exists x\in\nn \ \psi)$.  First take a world $w$ such that $\Diamond\exists x\in\nn \ \psi$ holds at $w$. Then take worlds $s,w'$ such that $R(w,s)$, $R(w,w')$, $\exists x\in\nn \ \psi$ holds at $w'$ and we want to show $\Diamond\exists x\in\nn \ \psi$ holds at $s$.  At $w'$ there is an $a\in D(w')$ such that $a\in I(\nn,w')$ and $\psi(a)$ holds at $w'$.  So, by Lemma~\ref{lem:sstability}, $\nn a\rightarrow\Box\nn a$ holds at $w'$ and by the induction hypothesis, $\psi(a)\rightarrow\Box\psi(a)$.  Let $s'$ be such that $R(s,s')$ and $R(w',s')$, such a world exists by directedness.  It follows that $\nn a$ and $\psi(a)$ hold at $s'$ and as $s'$ is accessible from $s$ we have proven $\Diamond\exists x\in\nn \ \psi$ holds at $s$. 
\end{proof}
%%%%%%%%%%%%%%%%%%%%%%%%%%%%%%%%

\noindent We can now prove that the modalized induction schema holds for all inductive formulas.  We do this by showing the more general result that induction holds for all stable formulas.

%LEMMA
%%%%%%%%%%%%%%%%%%%%%%%%%%%%%%%%
\begin{lem}\label{carl}
    If $\varphi$ is stable, then $$\vDash_{\mathsf{PI}}[\varphi(0)\wedge\Box\forall x,y\in\nn(\varphi(x)\wedge S(x,y)\rightarrow \varphi(y))]\rightarrow\Box\forall x\in\nn \ \varphi(x).$$
\end{lem}
%%%%%%%%%%%%%%%%%%%%%%%%%%%%%%%%

%PROOF
%%%%%%%%%%%%%%%%%%%%%%%%%%%%%%%%
\begin{proof}
     Let $w$ be a world. Further, we assume the antecedent of the induction schema holds so let $\varphi(0)$ and  $\Box\forall x,y\in\nn(\varphi(x)\wedge S(x,y)\rightarrow \varphi(y))$ hold at $w$.  Let $s$ be a world accessible from $w$ and let $a\in I(\nn,s)$.  We will show that $\varphi(a)$ at $s$.  If $a=0$ then, as $\varphi$ is stable, we are done so assume not. 
     
     As $a\in I(\nn,s)$, if we prove $\forall x,y(\varphi(x)\wedge\nn x\wedge S(x,y)\rightarrow\varphi(y)\wedge\nn y)$ and $\forall x(S(0,x)\rightarrow\varphi(x)\wedge\nn x)$ hold at $s$ then we have satisfied the antecedent of $S^+0a$ and so it follows that $\varphi (a)\wedge\nn a$ at $s$.  
     
     At $s$ we have $\forall x,y\in\nn(\varphi(x)\wedge S(x,y)\rightarrow \varphi(y))$.  We also have that if $x\in I(\nn,s),$ and $S(x,y)$ hold at $s$ then by Lemma~\ref{lem:nnante1} that $y\in I(\nn,s)$.  This proves $\forall x,y(\varphi(x)\wedge\nn x\wedge S(x,y)\rightarrow\varphi(y)\wedge\nn y)$ at $s$.
     
    From $a\in I(\nn,s)$ it follows that $0\in D(s)$.  Assume $x\in D(s)$ and $S0x$, as $0\in D(s)$ it follows by Lemma~\ref{lem:nnante2} that $x\in I(\nn,s)$.  It then follows by the stability of $\varphi$ that $\varphi(0)$ at $s$. As such we have the antecedent of $\forall x,y\in\nn(\varphi(x)\wedge S(x,y)\rightarrow \varphi(y))$ so we get $\varphi(x)$.  And from this it follows that $\forall x(S(0,x)\rightarrow\varphi(x)\wedge\nn x)$ holds at $s$.
\end{proof}  
%%%%%%%%%%%%%%%%%%%%%%%%%%%%%%%%

\noindent So we have proven the modalized induction axiom restricted to inductive formulas.  But we cannot prove modalized induction for all formulas in the language of potentially infinite models, as the following counterexample shows.  

%LEMMA
%%%%%%%%%%%%%%%%%%%%%%%%%%%%%%%%
\begin{lem}\label{lem:nonstablenoinduct} If $\varphi(x)$ is $\forall z(z=x)$, then
%EQUATION
%%%%%%%%%%%%%%%%%%%%%%%%%%%%%%%%
$$\nvDash_{\mathsf{PI}}[\varphi(0)\wedge\Box\forall x,y\in\nn(\varphi(x)\wedge S(x,y)\rightarrow \varphi(y))]\rightarrow\Box\forall x\in\nn \ \varphi(x).$$
%%%%%%%%%%%%%%%%%%%%%%%%%%%%%%%%
\end{lem}
%%%%%%%%%%%%%%%%%%%%%%%%%%%%%%%%

%PROOF
%%%%%%%%%%%%%%%%%%%%%%%%%%%%%%%%
\begin{proof}
It is sufficient to show there is a model and a world in the model where this statement is false.  Take the minimal model from Example~\ref{exmp:min} and world \textbf{0}, where $D(\mathbf{0})=\set{\mathbf{0}}$.  Clearly $\momodels{\mathbf{0}}{g}\forall z(z=0)$.  Let $w\in W$ be such that $R(\mathbf{0},w)$ and assume that for all $x,y\in I(\nn,w)$, that $\forall z(z=x)$ and $ S(x,y)$ hold at $w$.   As everything in the domain is equal to $x$ it follows that $y=x$ and so $\forall z(z=y)$ at $w$.  So $\momodels{\mathbf{0}}{g}\Box\forall x,y\in\nn(\forall z(z=x)\wedge S(x,y)\rightarrow \forall z(z=y))$.  But it does not follow that $\Box\forall x\in\nn \ \forall z(z=x)$, because $1\in W$  is a counterexample as $D(\mathbf{1})=\set{\mathbf{0},\mathbf{1}}$.
\end{proof}
%%%%%%%%%%%%%%%%%%%%%%%%%%%%%%%%

%%%%%%%%%%%%%%%%%%%%%%%%%%%%%%%%%%%%%%%%%%%%%%%%%%%%%%%%%%%%%%%%%%%%%%%%%%%%%%%%%%%%%%%%%%%%%%%%%%%%%%%%%%%%%%%%%%%%%%%%%%%%%%%%%%%%%%%%%%%%%%%%%%%%%%%%%%%%%%%%%
%%%%%%%%%%%%%%%%%%%%%%%%%%%%%%%%
%%%%%%%%%%%%%%%%%%%%%%%%%%%%%%%%
%%%%%%%%%%%%%%%%%%%%%%%%%%%%%%%%
\section{Proof of Theorem~\ref{thm:main}}\label{sec:interpertation}
%%%%%%%%%%%%%%%%%%%%%%%%%%%%%%%%
%%%%%%%%%%%%%%%%%%%%%%%%%%%%%%%%
%%%%%%%%%%%%%%%%%%%%%%%%%%%%%%%%

We now have almost all the pieces needed to prove Theorem~\ref{thm:main}.   However, before we do that we need to discuss what a translation and interpretation are in our setting because we are moving between logics.  

Intuitively, a \emph{translation} between two languages starts with instructions on how to rewrite atomic formulas in one language into the other language. It does not make any changes to the propositional connectives but can restrict the quantifiers to objects meeting some conditions.  In the current setting, however, we need a formal definition of what is to count as a translation when the underlying logics are different.  This notion should, at the very least, capture the Linnebo translation.   We offer the following definition as a minimal condition on any translation, though more will need to be done to ensure a widely applicable definition of translation and interpretation between logics.

%DEFINITION
%%%%%%%%%%%%%%%%%%%%%%%%%%%%%%%%
\begin{definition}\label{def:verifiabletranslation}
   Let $L_A$ and $L_B$ be two logics extending first-order predicate logic, defined by the languages $\mathcal{L}_A$ and $\mathcal{L}_B$ and derivability relations $\vdash_{L_A}$ and $\vdash_{L_B}$ respectively. 
   A \textit{generalised translation} is given by a recursive map $(\cdot)^{\mathcal{G}}:\mathcal{L}_A\rightarrow \mathcal{L}_B$ which preserves free variables and a domain formula $\delta(x)\in\mathcal{L}_B$, such that the map  is compositional on the propositional connectives and where for all unnested formulas\footnote{An unnested formula is one where the atomic subformulas of a formula contain at most one constant, function or relation (\cite[58]{Hodges1993}).  We only give conditions for unnested formulas.  So, for example, $Sxy$ and $+(x,y,z)$ are unnested but $S0x$ and $+(0,0,z)$ are nested.  Every formula is equivalent to an unnested one (\cite[p.~59, Cor~2.6.2]{Hodges1993}).  As such the translation can be expanded to unnested formulas using this equivalence. } $\varphi_1,\dots,\varphi_n,\psi$ containing free variables $x_1,\dots,x_m$ one has the following: 
   %EQUATION
   %%%%%%%%%%%%%%%%%%%%%%%%%%%%%%%%
   \begin{equation}\label{equ:gentrans}
       \varphi_1,\dots,\varphi_n\vdash_{L_A}\psi\Rightarrow \delta(x_1),\dots,\delta(x_m),\varphi_1^{\mathcal{G}},\dots,\varphi_n^{\mathcal{G}}\vdash_{L_B}\psi^{\mathcal{G}}
   \end{equation}
   %%%%%%%%%%%%%%%%%%%%%%%%%%%%%%%%
\end{definition}
%%%%%%%%%%%%%%%%%%%%%%%%%%%%%%%%

What we have done so far is an informal translation from the first-order language of arithmetic into the signature of the potentially infinite models. In Section~\ref{sec:lang} we showed how the atomic formulas could be translated.  Further, the modalized versions of the axioms of $\PAO$ proven in Sections~\ref{sec:q} and \ref{sec:iterp} are the translations of $\PAO$'s axioms via the translation found in Section~\ref{sec:lang} and the Linnebo translation for the quantifiers. 

While it has been set out in previous sections, for the sake of definiteness we here record the translation explicitly. We will call this translation $(\cdot)^\mathcal{F}$, as it is a Fregean translation.  Three things are worth noting before we lay out the translation.  The first is that the domain formula associated to this interpretation is $\nn$ from Definition~\ref{nn}.  The second is that the range of this translation is the inductive formulas from Definition~\ref{def:inductiveformulas}.  The third is that \ref{msuc}-\ref{mmultip} are the same definitions given in \ref{suc}, \ref{addit} and \ref{multip}.  We have not changed the definitions we are working with. Rather, we merely show how these definitions can be used to define the interpretation function $(\cdot)^\mathcal{F}$.  
%EQUATION GROUP
%%%%%%%%%%%%%%%%%%%%%%%%%%%%%%%%
\begin{align}
    0^\mathcal{F}\equiv & \#\varnothing, \\ 
    Sab^\mathcal{F}\equiv &\Diamond\exists G\exists u[Gu\wedge (b=\#G)\wedge(a=\#G\cup\set{u})],\label{msuc}\\
    +\!(a,b,c)^\mathcal{F}\equiv &\Diamond\exists X,Y(a{=}\#X\wedge b{=}\#Y\wedge c{=}\#X\cup Y\wedge X\cap Y=\varnothing),\\
%
    %SPLIT
    %%%%%%%%%%%%%%%%%%%%%%%%%%%%%%%%
    \begin{split}
    \quad \times\!(a,b,c)^\mathcal{F}\equiv&\Diamond\exists X,P[\#X=b\wedge \forall x\in X(\#\set{y\mid Pxy}=a)\wedge\\& \forall x,y\in X(x\neq y\rightarrow\set{z\mid Pxz}\cap\set{z\mid Pyz}=\varnothing)\wedge \#\bigcup_{x\in X}\set{y\mid Pxy}=c],
    \end{split}\label{mmultip}\\
    %%%%%%%%%%%%%%%%%%%%%%%%%%%%%%%%
%
    (\psi\wedge\chi)^{\mathcal{F}}\equiv&\psi^{\mathcal{F}}\wedge\chi^{\mathcal{F}},\\ 
    (\neg\psi)^{\mathcal{F}}\equiv&\neg\psi^{\mathcal{F}}, \\
    (\forall x\psi)^{\mathcal{F}}\equiv&\Box\forall x(\nn(x)\rightarrow\psi^{\mathcal{F}}),\\
    (\forall X^n\psi)^{\mathcal{F}}\equiv&\Box\forall X^n(\forall x_1,\dots, x_n(X^nx_1\dots x_n\rightarrow \nn(x_1)\wedge\dots \wedge \nn(x_n))\rightarrow\psi^{\mathcal{F}}). 
\end{align}
%%%%%%%%%%%%%%%%%%%%%%%%%%%%%%%%
   
\noindent To see that this is a generalised translation all that remains to be shown is that deduction is preserved by our translation. We need this result for both $\EPI$ and $\PI$.\footnote{Recall that we formalised $\PI$ in $\ACA$, and those interested in the nuts and bolts are directed to Appendix~\ref{ap:piintoapa0}.}  

%LEMMA
%%%%%%%%%%%%%%%%%%%%%%%%%%%%%%%%
\begin{lem}\label{lem:Visser}
    Let $\varphi_0,\dots,\varphi_n,\psi$ be unnested formulas in the language of $\PAO$ with free variables $v_0,\dots,v_m$, it follows that if $ \varphi_0,\dots,\varphi_n\vdash\psi$,  then $\nn(v_0),\dots,\nn(v_m),\varphi_0^\mathcal{F},\dots,\varphi_n^\mathcal{F}\vDash_{\mathsf{PI}}\psi^\mathcal{F}$. Further, it is $\PAO$-provable that if $ \varphi_0,\dots,\varphi_n\vdash\psi$ then $\ACA \vdash \mbox{``}\nn(v_0),\dots,\nn(v_m),$ $\varphi_0^\mathcal{F},\dots,$ $\varphi_n^\mathcal{F}$ $\vDash_{\mathsf{PI}}$ $\psi^\mathcal{F}\mbox{''}$.
\end{lem}
%%%%%%%%%%%%%%%%%%%%%%%%%%%%%%%%

\noindent The first part of this Lemma is similar to \textcite[Thm.~5.4.]{Linnebo2013}.  But he proves a version of this which does not restrict the quantifiers to a domain.  The modification to our case is simple and so we omit the proof.

On its own a translation is not very interesting. However, a translation is an \emph{interpretation} if the translations of the axioms of the interpreted theory can be proven in the interpreting theory.  

%DEFINITION
%%%%%%%%%%%%%%%%%%%%%%%%%%%%%%%%
\begin{definition}\label{def:verifiableinterpretation}
\noindent Let $\mathsf{T}_A$ and $\mathsf{T}_B$ be ${L}_A$ and ${L}_B$ theories respectively, where a theory is a set of sentences not necessarily closed under deduction.  A generalised translation $(\cdot)^{\mathcal{G}}:\lang_A\rightarrow\lang_B$ interprets $\mathsf{T}_A$ in $\mathsf{T}_B$, if for all $\mathcal{L}_A$ unnested sentences $\chi$: 
    %EQUATION
    %%%%%%%%%%%%%%%%%%%%%%%%%%%%%%%%
   \begin{equation}\label{equ:gentrans2}
       \mathsf{T}_A\vdash_{L_A}\chi\Rightarrow \mathsf{T}_B\vdash_{L_B}\chi^\mathcal{G}
   \end{equation}
   %%%%%%%%%%%%%%%%%%%%%%%%%%%%%%%%
   
\noindent It is a \textit{recursive interpretation} if the collection of $\lang_A$ and $\lang_B$ formulas are recursive, $\mathsf{T}_A$ and $\mathsf{T}_B$ are also recursive, as is $(\cdot)^\mathcal{G}$, and there are recursive maps from proofs to proofs which witness the truth of equations~(\ref{equ:gentrans}) and (\ref{equ:gentrans2}). If $\mathsf{T}$ extends $\mathsf{PA}^1$, then say that the interpretation is \emph{$\mathsf{T}$-verifiable} if the recursive functions are provably total in $\mathsf{T}$ and if the universal closures of the arithmetized versions of \ref{equ:gentrans} and \ref{equ:gentrans2} are provable in $\mathsf{T}$.
\end{definition}
%%%%%%%%%%%%%%%%%%%%%%%%%%%%%%%%

\noindent So, the proofs of Sections~\ref{sec:q} and \ref{sec:iterp} show our translation is an interpretation of $\PAO$ in $\EPI$.  However, to show it is an interpretation in $\PI$ a certain level of caution is needed because $\PI $ does not have a background derivability relation.  To resolve this, we take $\varphi_0,\dots,\varphi_n\vdash_{L_{PI}}\varphi$ to be $\ACA\vdash``\varphi_0,\dots,\varphi_n\vDash_{\mathsf{PI}}\varphi"$, where this is as defined in Appendix~\ref{ap:piintoapa0}. And, of course $\PI $ is just as defined in (\ref{equ:tpi}) of section \ref{sec:intro}, namely the set of sentences $\varphi$ such that  $\ACA\vdash``\vDash_{\mathsf{PI}}\varphi"$.  We then need to show the following:

%LEMMA
%%%%%%%%%%%%%%%%%%%%%%%%%%%%%%%%
\begin{lem}\label{lem:paotoaca} For all sentences $\varphi$ in the language of $\PAO$,
if $\PAO\vdash\varphi$ then $\ACA\vdash``\vDash_{\mathsf{PI}}\varphi^{\mathcal{F}}"$. Further, it is $\PAO$-provable that if $\PAO\vdash\varphi$ then $\ACA\vdash``\vDash_{\mathsf{PI}}\varphi^{\mathcal{F}}"$.
\end{lem}
%%%%%%%%%%%%%%%%%%%%%%%%%%%%%%%%

%PROOF
%%%%%%%%%%%%%%%%%%%%%%%%%%%%%%%%
\begin{proof}
    By Lemmas~\ref{phoibos}-\ref{multfunc2} and \ref{maire} and \ref{carl} we know that if $\varphi$ is an axiom of $\PAO$ then $\ACA\vdash``\vDash_{\mathsf{PI}}\varphi^{\mathcal{F}}"$.  Assume $\PAO\vdash\varphi$ not an axiom, then there are $n$ axioms of $\PAO$, $\varphi_0,\dots,\varphi_n$, such that $\varphi_0,\dots,\varphi_n\vdash\varphi$.  Then as we can always take the universal closure of axioms and $\varphi$ is a sentence it follows by Lemma~\ref{lem:Visser} that  $\ACA\vdash``\varphi_0^\mathcal{F},\dots,\varphi_n^\mathcal{F}\vDash_{\mathsf{PI}}\varphi^{\mathcal{F}}"$.  Given that the axioms are PI valid, it follows that $\ACA\vdash``\vDash_{\mathsf{PI}}\varphi^{\mathcal{F}}"$.
\end{proof}
%%%%%%%%%%%%%%%%%%%%%%%%%%%%%%%%

\noindent This final piece gives us the proof of:

%THEOREM 
%%%%%%%%%%%%%%%%%%%%%%%%%%%%%%%%
\begin{cthm}{\ref{thm:main}.ii.} { There is a generalised translation from the language of $\PAO$ to the second-order modal language with octothorpe that interprets $\PAO$ in $\PI$. Further, this is a $\PAO$-verifiable generalised interpretation.}
\end{cthm}
%%%%%%%%%%%%%%%%%%%%%%%%%%%%%%%%

\noindent To prove the first half of Theorem~\ref{thm:main} we need to define formulas that pick out the numbers in  $\PAO$ and $\EPI$.  In $\PAO$ let $\tau_0(x)\equiv (x=0)$ and $\tau_{n+1}(x)\equiv \exists y(\tau_n(y)\wedge Syx)$.  In $\EPI$ let $\sigma_0(x)\equiv (x=0)$ and $\sigma_{n+1}(x)\equiv\Diamond\exists y\in\nn(\sigma_n(y)\wedge Syx)$.  Note that $(\tau_0(x))^\mathcal{F}\equiv(x=0)^\mathcal{F}\equiv\sigma_0(x)$ and $(\tau_{n+1}(x))^\mathcal{F}\equiv(\exists y(\tau_n(y)\wedge Syx))^\mathcal{F}\equiv\Diamond\exists y\in\nn((\tau_n(y))^\mathcal{F}\wedge Syx)\equiv\sigma_{n+1}(x)$.  With this we can state the following preliminary Lemma; we omit the proof which is long but not illuminating:

%LEMMA
%%%%%%%%%%%%%%%%%%%%%%%%%%%%%%%%
\begin{lem}\label{lem:nntopi}
For every $k\geq 0$ and every unnested formula $\theta(x_1, \dots, x_k)$ in the signature of $\PAO$ and every $k$-tuple of natural numbers $n_1, \dots, n_k$ one has that :
        %EQUATION
        %%%%%%%%%%%%%%%%%%%%%%%%%%%%%%%%
        \begin{equation}
            \nn\models \theta(n_1, \dots, n_k)\quad\Longrightarrow\quad\vDash_{\mathsf{PI}}\forall x_1, \dots, x_k\in\nn (\bigwedge_{i=1}^k \sigma_{n_i}(x_i) \rightarrow \theta^{\mathcal{F}}(x_1, \dots, x_k)) \label{lem:nntopi:1}
        \end{equation}
        %%%%%%%%%%%%%%%%%%%%%%%%%%%%%%%%
In the case of $k=0$, this is to say: for every unnested sentence $\theta$ in the signature of $\PAO$ one has that
        %EQUATION
        %%%%%%%%%%%%%%%%%%%%%%%%%%%%%%%%
        \begin{equation}
            \nn\models \theta\quad\Longrightarrow\quad\vDash_{\mathsf{PI}} \theta^{\mathcal{F}}\label{lem:nntopi:2}
        \end{equation}
        %%%%%%%%%%%%%%%%%%%%%%%%%%%%%%%%
\end{lem}
%%%%%%%%%%%%%%%%%%%%%%%%%%%%%%%%

\noindent Theorem~\ref{thm:main}.i follows from (\ref{lem:nntopi:2}) of Lemma~\ref{lem:nntopi}.  This give us our proof of:

%THEOREM 
%%%%%%%%%%%%%%%%%%%%%%%%%%%%%%%%
\begin{cthm}{\ref{thm:main}.i.} 
{ There is a generalised translation from the language of $\PAO$ to the second-order modal language with octothorpe that interprets $\TAO$ in $\EPI$.}
\end{cthm}
%%%%%%%%%%%%%%%%%%%%%%%%%%%%%%%%

%%%%%%%%%%%%%%%%%%%%%%%%%%%%%%%%%%%%%%%%%%%%%%%%%%%%%%%%%%%%%%%%%%%%%%%%%%%%%%%%%%%%%%%%%%%%%%%%%%%%%%%%%%%%%%%%%%%%%%%%%%%%%%%%%%%%%%%%%%%%%%%%%%%
%%%%%%%%%%%%%%%%%%%%%%%%%%%%%%%%
%%%%%%%%%%%%%%%%%%%%%%%%%%%%%%%%
%%%%%%%%%%%%%%%%%%%%%%%%%%%%%%%%
\section{Proof of Theorem~\ref{thm:main2}}\label{sec:prmain2}
%%%%%%%%%%%%%%%%%%%%%%%%%%%%%%%%
%%%%%%%%%%%%%%%%%%%%%%%%%%%%%%%%
%%%%%%%%%%%%%%%%%%%%%%%%%%%%%%%%

It has been shown by \textcite[\S 7]{Linnebo2019-ll} that the Linnebo translation cannot interpret comprehension because modalized comprehension requires the existences of a set of all possibly existing things. However, this leaves open the question of whether there is a different translation which can interpret $\PAT$.  Here we will demonstrate that there is no translation from $\TAT$ to $\EPI$ nor from $\PAT$ to $\PI$ by proving Theorem~\ref{thm:main2}, our second main theorem.  The first part of Theorem~\ref{thm:main2} follows from relatively simple Tarskian considerations:

%THEOREM (MANUAL AS ABOVE)
%%%%%%%%%%%%%%%%%%%%%%%%%%%%%%%%
\begin{cthm}{\ref{thm:main2}.i}  {
There is no generalised translation from the language of $\PAT$ to the second-order modal language with octothorpe that interprets $\TAT$ in $\EPI$.}
\end{cthm}
%%%%%%%%%%%%%%%%%%%%%%%%%%%%%%%%

%PROOF
%%%%%%%%%%%%%%%%%%%%%%%%%%%%%%%%
\begin{proof}
    Assume for a contradiction that there is an interpretation $(\cdot)^\mathcal{G}$ that interprets $\TAT$ in $\EPI$. Note that as $\TAT$ is complete it follows that this is a faithful interpretation; i.e.\@ if $\vDash_{\mathsf{PI}}\varphi^\mathcal{G}$ then $\nn\vDash\varphi$.  As  $\EPI$ is $\Pi^1_1$-definable it follows that there is a predicate $P$ such that for all $\varphi$ in the second-order modal language with octothorpe we have $\vDash_{\mathsf{PI}}\varphi$ if and only if $\nn\vDash P(``\varphi")$.  (Here we use quotation marks for G\"odel numbering for both the language of $\PAT$ and the second-order modal language with octothorpe.) But then as generalised translations are recursive we can represent $(\cdot)^\mathcal{G}$ in $\nn$ as $g$. It follows that $P(g(``\psi"))$, where $\psi$ is in the language of $\PAT$, is a truth predicate for $\TAT$. But this contradicts Tarski's theorem. 
\end{proof}
%%%%%%%%%%%%%%%%%%%%%%%%%%%%%%%%

The proof of the second part of the theorem is trickier and requires G\"odelian considerations. Recall the definition of $\mathsf{T}$-verifiable generalised translation and interpretation from Definitions~\ref{def:verifiabletranslation} and \ref{def:verifiableinterpretation} in Section~\ref{sec:interpertation}.  There we proved that we have a $\PAO$-verifiable interpretation of $\PAO$ in $\PI$ by Lemma~\ref{lem:paotoaca}. Given that we defined $\PI \vdash\varphi$ as $\ACA\vdash\text{``$\vDash_{\mathsf{PI}}\varphi$"}$, that is $\PAO\vdash\forall\varphi[``\PAO\vdash\varphi"\rightarrow``\ACA\vdash\text{``}\vDash_{\mathsf{PI}}\varphi^\mathcal{F}{"}"]$.    Here we show that there is no $\PAT$-verifiable interpretation of $\PAT$ in $\PI$.  We can write this as: there is no generalised translation $(\cdot)^\mathcal{G}$ from the language of $\PAT$ to the second-order modal language with octothorpe such that $\PAT\vdash\forall\varphi[``\PAT\vdash\varphi"\rightarrow``\ACA\vdash\text{``}\vDash_{\mathsf{PI}}\varphi^\mathcal{G}""$].

%THEOREM 
%%%%%%%%%%%%%%%%%%%%%%%%%%%%%%%%
\begin{cthm}{\ref{thm:main2}.ii}  { There is no  generalised translation from the language of $\PAT$ to the second-order modal language with octothorpe that $\PAT$-verifiably interprets $\PAT$ in $\PI$.}
\end{cthm}
%%%%%%%%%%%%%%%%%%%%%%%%%%%%%%%%

%PROOF
%%%%%%%%%%%%%%%%%%%%%%%%%%%%%%%%
\begin{proof}
    The systems $\PICAk{k}$ are subsystems of $\PAT$ that have comprehension for $\Pi^1_k$ formulas. As proofs are finite and so can only use finitely many instances of the comprehension schema any interpretation which is $\PAT$-verifiable will also be $\PICAk{k}$-verifiable for some $k\geq 1$.  Let $\varphi_1,\dots,\varphi_n$ be a finite axiomatisation of $\PICAk{k}$ for some $k\geq1$ (\cite[pp.~303, 311-2]{Simpson2009}).    We will show, from the assumption that there is a $\PICAk{k}$-verifiable translation $(\cdot)^\mathcal{G}$ from the languge of $\PAT$ to the second-order modal language with octothorpe that interprets $\PAT$ in $\PI$, that $\PICAk{k}$ proves its own consistency. This contradicts G\"odel's second incompleteness theorem and so shows that no such $(\cdot)^\mathcal{G}$ can exist.
    
    Note that $\PAT\vdash \varphi_1,\dots,\varphi_n$ as all $\PICAk{k}$ are subsystems of $\PAT$.  We are assuming that $(\cdot)^\mathcal{G}$ interprets $\PAT$ in $\PI$, so it follows that $\ACA\vdash``\vDash_{\mathsf{PI}}\varphi_1^\mathcal{G},\dots,\varphi_n^\mathcal{G}"$. Let $\mathcal{A}$ be a model of $\PICAk{k}$ for some $k$.  So, we have $\mathcal{A}\vDash``\vDash_{\mathsf{PI}}\varphi_1^\mathcal{G},\dots,\varphi_n^\mathcal{G}"$.    If $\mathcal{M}$ is the minimal model from Example~\ref{exmp:min} relative to $\mathcal{A}$ then we have then we have $\mathcal{A}\vDash``\mathcal{M}\vDash\varphi_1^\mathcal{G},\dots,\varphi_n^\mathcal{G}"$. 
    
    Now we show that $\mathcal{A}\vDash\neg Prv_{\varphi_1,\dots,\varphi_n}(\psi\wedge\neg\psi)$, that is the consistency of $\PICAk{k}$.  Assume for a contradiction that  $\mathcal{A}\vDash \exists\pi Prf_{\varphi_1,\dots,\varphi_n}(\pi,\psi\wedge\neg\psi)$.  Then as $(\cdot)^\mathcal{G}$ is a $\PICAk{k}$-verifiable interpretation it follows $\mathcal{A}\vDash Prf_{\ACA}(\pi^\mathcal{G},``\varphi_1^\mathcal{G},\dots,\varphi_n^\mathcal{G}\vDash_{\mathsf{PI}}\psi^\mathcal{G}\wedge\neg\psi^\mathcal{G}{}'')$.  
    
    Recall that $\PICA$ proves $\Sigma^1_1$-reflection for $\ACA$ (cf. \textcite{Simpson2009} Theorem~VII.6.9.(4) p.~298 and Theorem~VII.7.6.(1)~p.~305). As $\PICA\subseteq\PICAk{k}$, this means that for any $\Pi^1_1$ statement $\psi$ we know $\PICAk{k}$ proves $Prv_{\ACA}(\psi)\rightarrow \psi$. For all $\psi$, we know that $``\vDash_{\mathsf{PI}}\psi"$ is $\Pi^1_1$ and similarly for the local derivability relation (see Appendix~\ref{ap:piintoapa0}).  It follows that $\mathcal{A}\vDash``\varphi_1^\mathcal{G},\dots,\varphi_n^\mathcal{G}\vDash_{\mathsf{PI}}\psi^\mathcal{G}\wedge\neg\psi^\mathcal{G}{"}$ and as $\mathcal{A}\vDash``\mathcal{M}\vDash\varphi_1^\mathcal{G},\dots,\varphi_n^\mathcal{G}{"}$.  It follows that  $\mathcal{A}\vDash``\mathcal{M}\vDash\psi^\mathcal{G}\wedge\neg\psi^\mathcal{G}{"}$.  And so $\mathcal{A}\vDash``\mathcal{M}\models \psi^G"$ and $\mathcal{A}\vDash``\mathcal{M}\models \neg (\psi^G)"$.  
\end{proof}
%%%%%%%%%%%%%%%%%%%%%%%%%%%%%%%%

\noindent  We have now shown the two main results set out in the introduction.
\section{Conclusion}

We started with the worry that Hume's Principle had only infinite models and so any claim that it was analytic would mean that the claim that there are infinitely many objects is analytic.  This worry has been noted before in the literature on neo-logicism, but little has been done to address it.  \textcite{Hale2001} state that without this the neo-logicist project cannot even get off the ground:

%QUOTE
%%%%%%%%%%%%%%%%%%%%%%%%%%%%%%%%
\begin{quote}
    To require of an acceptable abstraction that it should not be (even) weakly inflationary [that is require a countable infinity] would stop the neo-Fregean project dead in its tracks, before it even got moving (as it were). It will be clear that I think there is no good ground to impose such a requirement, and I shall not discuss it further.
    \unskip\hspace*{1em plus 1fill}%
   \nolinebreak[3]\hspace*{\fill}\mbox{(\cite[417--8]{Hale2001})}
\end{quote}
%%%%%%%%%%%%%%%%%%%%%%%%%%%%%%%%

\noindent In this paper we have explored the potentially infinite as one way to address this worry.  The move to the potentially infinite does not rid us of posited infinities. We still require there to be an infinity of worlds and an infinity of objects across the worlds.  But these infinities are less metaphysically questionable.    So, for example, while Putnam and Hodes objected to the positing of actual infinities they allowed for possible infinities.  And one could always try to further avoid the commitment by adopting an instrumentalist attitude towards the metatheory.

We have shown that the theory of potentially infinite models interprets first-order Peano arithmetic or first-order true arithmetic, depending on the strength of our meta-language.  But we cannot interpret the equivalent second-order arithmetic theory.  The difficulty seems to be the non-existence of a set of all the numbers across all the worlds. As our models are supposed to capture the idea of the potential infinite, we do not want the set of all the numbers across all the worlds to exist.   It makes sense that the potential infinite does not capture the infinite progression of the natural numbers as well as actual infinity and this might go some way to explaining why we get the weaker first-order theory.  

This allows a fuller understanding of the role of the potentially infinite in the foundation of mathematics.  Unlike Hodes, we see that a certain amount of mathematics can be recovered, though some other story would need to be told about more advanced mathematics. It also offers evidence that the ontological commitments that come with Hume's Principle, and which make some reject the claim that its truth is analytic, cannot be avoided by moving to the modal setting if one wants full second-order Peano arithmetic.  For in weakening our ontological commitments, we also weakened the mathematical theory which we can recover.

%%%%%%%%%%%%%%%%%%%%%%%%%%%%%%%
%%%%%%%%%%%%%%%%%%%%%%%%%%%%%%%
%%%%%%%%%%%%%%%%%%%%%%%%%%%%%%%
\begin{appendices}
%%%%%%%%%%%%%%%%%%%%%%%%%%%%%%%
%%%%%%%%%%%%%%%%%%%%%%%%%%%%%%%
%%%%%%%%%%%%%%%%%%%%%%%%%%%%%%%
\section{Formal Theories}\label{ap:theories}
%%%%%%%%%%%%%%%%%%%%%%%%%%%%%%%%
%%%%%%%%%%%%%%%%%%%%%%%%%%%%%%%%
%%%%%%%%%%%%%%%%%%%%%%%%%%%%%%%%

Here we will spell out the theories other than $\EPI$ and $\PI$ which are used in the proofs above.  Unlike $\EPI$ and $\PI$ none of these are modal theories, however, most are second-order theories.

The weakest theory we consider is first-order Robinson's $\Q$.  For a more complete reference see, for example, \textcite[28]{Hajek1998}.

%DEFINITION
%%%%%%%%%%%%%%%%%%%%%%%%%%%%%%%%%%%%%%%%%%%%%%%%%%%%%%%%%%%%%%%%%%%%%%%%%%
\begin{definition}
$\Q$ is the usual formalization of Robinson's arithmetic. It consists of the universal closure of the following axioms:
%EQUATION GROUP
%%%%%%%%%%%%%%%%%%%%%%%%%%%%%%%%
\begin{align}
    &s(x)\neq 0; && \text{(Q1)} &
    &s(y)=s(z)\rightarrow y=z\label{ax:q2}\tag{Q2};\\
    &x+0=x; && \text{(Q3)} &
    &x+s(y)=s(x+y)\label{ax:q4}\tag{Q4};\\
    &x\times0=0; && \text{(Q5)} &
    &x\times s(y)=(x\times y)+y.\label{ax:q6}\tag{Q6}
\end{align}
%%%%%%%%%%%%%%%%%%%%%%%%%%%%%%%%
\end{definition}
%%%%%%%%%%%%%%%%%%%%%%%%%%%%%%%%%%%%%%%%%%%%%%%%%%%%%%%%%%%%%%%%%%%%%%%%%%

\noindent Note that in the body of the text we do not use this formulation but rather one with relations instead than functions.\footnote{We use a capital $S$ for the relational successor and lower case $s$ for the functional.} We have offered this formulation for readability. The relation formulation gives you the obvious translation of the above, plus an additional 6 axioms ensuring that the relations $S,+,\times$ are the graphs of functions.

We also consider the extensions of $\Q$ to $\PAO$ by the addition of the first-order induction schema, and $\PAT$ by the addition of the second-order induction axiom and Comprehension Schema.  $\PAO$ is a first-order theory, but $\PAT$ is a second-order theory.

%DEFINITION
%%%%%%%%%%%%%%%%%%%%%%%%%%%%%%%%%%%%%%%%%%%%%%%%%%%%%%%%%%%%%%%%%%%%%%%%%%
\begin{definition}
$\PAO$ is $\Q$ plus the induction schema, where $\varphi$ is a first-order formula:
%EQUATION
%%%%%%%%%%%%%%%%%%%%%%%%%%%%%%%%
\begin{equation}
    (\varphi(0) \wedge\forall x(\varphi x\rightarrow \varphi (s(x))))\rightarrow\forall x \varphi (x)\label{ax:is}\tag{Induction Schema (IS)}
\end{equation}
%%%%%%%%%%%%%%%%%%%%%%%%%%%%%%%%

\noindent $\PAT$ is $\Q$ plus the induction axiom and Comprehension Schema:
%EQUATION
%%%%%%%%%%%%%%%%%%%%%%%%%%%%%%%%
\begin{equation}
    \forall P[(P0 \wedge\forall x(Px\rightarrow P(s(x))))\rightarrow\forall x Px]\tag{Induction Axiom (IS)}
\end{equation}
%%%%%%%%%%%%%%%%%%%%%%%%%%%%%%%%
%EQUATION
%%%%%%%%%%%%%%%%%%%%%%%%%%%%%%%%
\begin{equation}
    \forall\Bar{y},\Bar{Y}\exists X\forall x(X(x)\leftrightarrow \varphi(x,\Bar{y},\Bar{Y}))\tag{Comprehension Schema (CS)}
\end{equation}
%%%%%%%%%%%%%%%%%%%%%%%%%%%%%%%%

\noindent In the Comprehension Schema $\varphi$ can be any formula of the language of $\PAT$ in which $X$ does not occur free.
\end{definition}
%%%%%%%%%%%%%%%%%%%%%%%%%%%%%%%%%%%%%%%%%%%%%%%%%%%%%%%%%%%%%%%%%%%%%%%%%%
\noindent Again in the body of the text we use the natural adaptation to the setting of relations rather than functions.  There are also two theories we use that are second-order and between $\PAT$ and $\PAO$ in strength. They both restrict comprehension. So, we first need to define the formulas we restrict to:

%DEFINITION
%%%%%%%%%%%%%%%%%%%%%%%%%%%%%%%%%%%%%%%%%%%%%%%%%%%%%%%%%%%%%%%%%%%%%%%
\begin{definition}(\cite[I.3.1, p.~6]{Simpson2009}) An {\it Arithmetical formula} is a formula in the language of $\PAT$ which does not contain any set quantifiers, though it may contain free set variables.
\end{definition}
%%%%%%%%%%%%%%%%%%%%%%%%%%%%%%%%%%%%%%%%%%%%%%%%%%%%%%%%%%%%%%%%%%%%%%%

\noindent With this we can state $\ACA$:

%DEFINITION
%%%%%%%%%%%%%%%%%%%%%%%%%%%%%%%%%%%%%%%%%%%%%%%%%%%%%%%%%%%%%%%%%%%%%%%%%%%%%
\begin{definition}(\cite[I.3.2, p.~7]{Simpson2009}) $\ACA$ is $\Q$ plus the Induction Axiom and Arithmetical Comprehension:
%EQUATION
%%%%%%%%%%%%%%%%%%%%%%%%%%%%%%%%
\begin{equation}
    \forall\Bar{y},\Bar{Y}\exists X\forall x(X(x)\leftrightarrow \varphi(x,\Bar{y},\Bar{Y}))\label{ax:acs}\tag{Arithmetical Comprehension Schema (ACS)}
\end{equation}
%%%%%%%%%%%%%%%%%%%%%%%%%%%%%%%%

\noindent Where $\varphi$ has to be an arithmetical formula and $X$ may not occur free.
\end{definition}
%%%%%%%%%%%%%%%%%%%%%%%%%%%%%%%%%%%%%%%%%%%%%%%%%%%%%%%%%%%%%%%%%%%%%%%%%%%%%%%%%

\noindent Note that as every formula of $\PAO$ is arithmetical, and $\ACA$ contains the second-order induction axiom, every instance of the first-order induction schema is provable in $\ACA$.

The next theories of arithmetic to be considered here are the $\PICAk{k}$ which are used in the proof of Theorem~\ref{thm:main2}.  To define this theory, we first need to define $\Pi_k^1$ (and $\Sigma^1_k$) formulas:

%DEFNITION
%%%%%%%%%%%%%%%%%%%%%%%%%%%%%%%%%%%%%%%%%%%%%%%%%%%%%%%%%%%%%%%%%%%%%%%%%%%%%%%%%
\begin{definition}(\cite[I.5.1, p.~16]{Simpson2009})
A {\it $\Pi^1_1$ formula} is a formula in the language of $\PAT$ of the form $\forall X_1,\dots,X_n\varphi$ where $X_1,\dots,X_n$ are set variables and $\varphi$ is an arithmetical formula.

\noindent A {\it $\Sigma^1_1$ formula} is a formula in the language of $\PAT$ of the form $\exists X_1,\dots,X_n\varphi$ where $X_1,\dots,X_n$ are set variables and $\varphi$ is an arithmetical formula.

\noindent A {\it $\Pi^1_k$ formula} is a formula in the language of $\PAT$ of the form $\forall X_1,\dots,X_n\varphi$ where $X_1,\dots,X_n$ are set variables and $\varphi$ is a $\Sigma^1_{k-1}$ formula.

\noindent A {\it $\Sigma^1_k$ formula} is a formula in the language of $\PAT$ of the form $\exists X_1,\dots,X_n\varphi$ where $X_1,\dots,X_n$ are set variables and $\varphi$ is $\Pi^1_{k-1}$ formula.
\end{definition}
%%%%%%%%%%%%%%%%%%%%%%%%%%%%%%%%%%%%%%%%%%%%%%%%%%%%%%%%%%%%%%%%%%%%%%%%%%%%%%%%%

\noindent The definition of $\PICAk{k}$ is much like the definition of $\ACA$, except that the restriction on the comprehension axiom is broadened to include all $\Pi_k^1$ formulas:

%DEFINITION
%%%%%%%%%%%%%%%%%%%%%%%%%%%%%%%%%%%%%%%%%%%%%%%%%%%%%%%%%%%%%%%%%%%%%%%%%%%%%%%%%
\begin{definition}(\cite[I.5.2, p.~17]{Simpson2009})
$\PICAk{k}$ is $\Q$ plus the Induction Axiom and $\Pi^1_k$ Comprehension:
%EQUATION
%%%%%%%%%%%%%%%%%%%%%%%%%%%%%%%%
\begin{equation}
    \forall\Bar{y},\Bar{Y}\exists X\forall x(X(x)\leftrightarrow \varphi(x,\Bar{y},\Bar{Y}))\label{ax:pics}\tag{$\Pi^1_k$ Comprehension Schema ($\Pi^1_k$CS)}
\end{equation}
%%%%%%%%%%%%%%%%%%%%%%%%%%%%%%%%

\noindent Where $\varphi$ has to be a $\Pi_k^1$ formula and $X$ may not occur free.
\end{definition}
%%%%%%%%%%%%%%%%%%%%%%%%%%%%%%%%%%%%%%%%%%%%%%%%%%%%%%%%%%%%%%%%%%%%%%%%%%%%%%%%%

We can define the intended model of these theories.  Let $\nn^1$ be $\set{\omega,0,s,+,\times}$ where each term is interpreted as it is in the metatheory and $\nn^2$ be $\nn^1$ with $\powerset(\omega^n)$ as the domain of the second-order quantifiers.  $\nn^1$ is the intended model of $\Q$ and $\PAO$, while $\nn^2$ is the intended model of $\PAT$, $\ACA$, and $\PICAk{k}$ for all $k$.  As is well known, by G\"odel's incompleteness theorems none of the theories we have seen so far are complete.  We can define the complete theories of these models:

%defnition
%%%%%%%%%%%%%%%%%%%%%%%%%%%%%%%%
\begin{definition}
Let $\TAO$ be $\set{\varphi\mid \nn^1\vDash\varphi}$ and $\TAT$ be $\set{\varphi\mid \nn^2\vDash\varphi}$.
\end{definition}
%%%%%%%%%%%%%%%%%%%%%%%%%%%%%%%%

For the sake of completeness, we here define Hume's Principle ($\HP$).  This system is second-order also and consists of the cardinality principle displayed in Equation~\ref{hp} on page~\pageref{hp}, the full Comprehension Schema, as in $\PAT$, and full comprehension for binary relations:
%EQUATION
%%%%%%%%%%%%%%%%%%%%%%%%%%%%%%%%
\begin{equation}
    \forall\Bar{y},\Bar{Y}\exists X\forall x,z(X(x,z)\leftrightarrow \varphi(x,z,\Bar{y},\Bar{Y}))\label{ax:bcs}\tag{Binary Comprehension Schema (BCS)}    
\end{equation}
%%%%%%%%%%%%%%%%%%%%%%%%%%%%%%%%

\noindent Comprehension for binary relations is required because the definition of $\HP$ quantifies over bijections and when spelt out fully this turns out to be the claim that there is a second-order binary relation which is the graph of a bijection between the two sets.

%%%%%%%%%%%%%%%%%%%%%%%%%%%%%%%%
%%%%%%%%%%%%%%%%%%%%%%%%%%%%%%%%
%%%%%%%%%%%%%%%%%%%%%%%%%%%%%%%%
\section{Formal definition of $\PI$}\label{ap:piintoapa0}
%%%%%%%%%%%%%%%%%%%%%%%%%%%%%%%%
%%%%%%%%%%%%%%%%%%%%%%%%%%%%%%%%
%%%%%%%%%%%%%%%%%%%%%%%%%%%%%%%%

In the introduction we gave $\PI$ as the set $ \set{\varphi\mid \ACA\vdash\text{`$\vDash_{\mathsf{PI}}\varphi$'}}$. Here we will layout explicitly what we mean by defining the arithmetization of $\vDash_{\mathsf{PI}}$ in $\ACA$.  

It is importaint to note that the second-order variables in $\PI$ are taken to first-order variables in $\ACA$.  If all the first-order variables of $\PI$ are of the form $x_i$ and all the second-order variables of $\PI$ are of the form $Y_j$ then let all the first-order variables of $\ACA$ be of the form $x_i$ and $Y_j$, and the second-order variables of $\ACA$ be of the form $Z_v$.  In practice we will not stick to this strict distinction, but it can always be implemented by renaming the variables.

We do not restrict the domain of the first-order variables of $\PI$; there is no need to pick out a subset of the domain of a model of $\ACA$.    However, the second-order variables of $\PI$ need to be restricted to codes for finite sets of numbers ordered by strict less than.  This isn't difficult, we can simply borrow the coding found in the proof of incompleteness. A more complete explication can be found in \textcite[Ch. 2.2]{Simpson2009}. The second-order variables are required to be to some sequence $\pi(0)^{n_0}+\dots+\pi(m)^{n_m}$ where $\pi(i)$ gives the $i$th prime and $n_0<n_1<\dots<n_m$.  Let $Seq(Y)$ be the name of the relation that ensures $Y$ has the above properties.  Further, let $nSeq(Y)$ mean that $Y$ codes $n$-tuples of numbers.  We will use this to code relations and relational variables. If $x$ is the number of a sequence then let $[x]_i$ be the $i$th element and $ln(x)$ is the length of $x$.

We want to code PI models as sets of natural numbers.  We know that we can always combine countably many countably infinite sets (just code $n$ a member of the $i$th set as $2^i+3^n$).  As such we will just show how to code $W,R,D,\#,\mathbf{a}$ as separate sets of natural numbers.  Further, with $R,D,\#,\mathbf{a}$ we will talk about pairs $(x,y)$, this should be understood as standing for the code $2^x+3^y$.

%LIST
%%%%%%%%%%%%%%%%%%%%%%%%%%%%%%%%
\begin{enumerate}[label={(B.\arabic*)}]
    \item \label{con:1}Let $W$ be infinite ($\forall x\in W\exists y\in W (y>x)$),\footnote{Recall that our definition demanded that our set of worlds be countable. We cannot capture this in $\ACA$ in the sense that $\ACA$ has none standard models but we will have that we do not have more worlds than $\ACA$ thinks there are natural numbers, which is sufficent for the role this plays in the proofs.}
    \item\label{con:2} let $R$ be such that \begin{enumerate}
        \item for all $(i,j)\in R$ we have that $i,j\in W$,
        %\item $\forall x\in W\exists y\in W(x\neq y\wedge R(x,y))$,
        %\item $\exists x\in W\forall y\in W (x\neq y\rightarrow\neg R(y,x)$ (the order has an initial element),
        \item $\forall x\in W \ R(x,x)$ (reflexive),
        \item $\forall x,y,z\in W (R(x,y)\wedge R(y,z)\rightarrow R(x,z))$ (transitive),
        \item $\forall x,y\in W(R(x,y)\wedge R(y,x)\rightarrow x=y)$ (anti-symmetric),
        \item $\forall x,y\in W \ \exists z\in W(R(x,z)\wedge R(y,z))$ (directed),
    \end{enumerate} 
    \item let $D$ be such that 
    \begin{enumerate}
        \item $D(w,Y)$ implies that $w\in W$ and $Seq(Y)$, 
        \item $\forall w\in W\exists Y\in Seq (D(w,Y)\wedge ln(Y)>0)$ (every world has at least one element),
        \item $D$ is the graph of a function from $W$ to $Seq$,
        \item if $R(i,j)$ and $i\neq j$ and $D(i,X)$ and $D(j,Y)$ then $\exists u\forall v ([X]_v\neq[Y]_u)$ (there is something in $Y$ not in $X$) and $\forall v<ln(X)\exists u([X]_v=[Y]_u)$ (everything in $X$ is in $Y$),
    \end{enumerate}  
    \item \label{con:4}let $\mathbf{a}$ be such that for each $n$ there is exactly one $x$ such that $\mathbf{a}(n,x)$ and if $\mathbf{a}(n,x)$ and $\mathbf{a}(m,x)$ then $n=m$, we then define $\#(Y,x)$ as $Seq(Y)\wedge\mathbf{a}({ln(Y)},x)$.
\end{enumerate}   
%%%%%%%%%%%%%%%%%%%%%%%%%%%%%%%%

Given a set of numbers $\mathcal{M}$ we will write $\mathcal{M}\in PIM$ to signify the set meets \ref{con:1}--\ref{con:4}. We define $sb$ (subset) as follows $Y\in sb(X)$ iff $Seq(Y)\wedge\forall i<ln(Y)\exists j ([X]_j=[Y]_i)$. In defining the arithmetisation note that we add free-variables for the model and the world, we will use $W_M,R_M,D_M,\#_M$, but these can be defined in terms of the model.  So, if $\varphi$ is a formula in the modal second-order language with octothorpe we translate it to some $\psi(w,W_M,R_M,D_M,\#_M)$ in the language of arithmetic. We define the arithmetisation as follows:
%EQUATION GROUP
%%%%%%%%%%%%%%%%%%%%%%%%%%%%%%%%
\begin{align}
    (x_i=x_j)^*\equiv& x_i=x_j\\
    \quad \ \ \ (x_i=\#Y_j)^*\equiv& \#_M(Y_j,x_i)\\
    (Y_jx_i)^*\equiv& \exists u(x_i=[Y_j]_u)\\
    (\forall x\varphi)^*\equiv& \forall x (\exists Y\in Seq(D_M(w,Y)\wedge \exists u (x=[Y]_u))\rightarrow (\varphi)^*)\\
    (\forall Y\varphi)^*\equiv& \forall Y\in Seq(\exists X\in Seq(D_M(w,X)\wedge Y\in sb(X))\rightarrow (\varphi)^*)\\
    (\forall P^n\varphi)^*\equiv& \forall P^n\in nSeq\\&\hfill(\exists X\in Seq(D_M(w,X)\wedge \forall(x_1,\dots,x_n)\in P^n (\bigwedge_{1\leq i\leq n}\exists j [X]_j=x_i))\rightarrow (\varphi)^*)\notag\\
    (\Box\varphi)^*\equiv& \forall s\in W_M(R_M(w,s)\rightarrow (\varphi)^*[w/s])
\end{align}
%%%%%%%%%%%%%%%%%%%%%%%%%%%%%%%%

\noindent where we commute over the logical connectives.  This means that every formula arithmetised is arithmetical as defined in Appendix~\ref{ap:theories}. For example, $\Box\forall v\Diamond\exists Z(v=\#Z)$ becomes 
%EQUATION
%%%%%%%%%%%%%%%%%%%%%%%%%%%%%%%%
\begin{multline}
    \forall s\in W_M(R_M(w,s)\rightarrow 
    \forall v (\exists Y(D_M(s,Y)\wedge \exists u (v=[Y]_u)\rightarrow \qquad\qquad\qquad\qquad \\
    \exists s'\in W_M(R_M(s,s')\wedge
    \exists Z\in Seq(\exists X(D_M(w,X)\wedge Z\in sb(X)\wedge
    \#_M(Z,v))))).
\end{multline}
%%%%%%%%%%%%%%%%%%%%%%%%%%%%%%%%

Note $`\vDash_{\mathsf{PI}}\varphi$' means $\forall M\in PIM\forall w\in W_M(\varphi)^*$.  It follows that this is then a $\Pi^1_1$ formula.  Hence, if one were proceeding very formally, we would define $\PI$ as the set of all the $\varphi$ such that $\ACA\vdash\forall M\in PIM\forall w\in W_M(\varphi)^*$.

\end{appendices}

%%%%%%%%%%%%%%%%%%%%%%%%%%%%%%%%
%%%%%%%%%%%%%%%%%%%%%%%%%%%%%%%%
%%%%%%%%%%%%%%%%%%%%%%%%%%%%%%%%
\printbibliography
%%%%%%%%%%%%%%%%%%%%%%%%%%%%%%%%
%%%%%%%%%%%%%%%%%%%%%%%%%%%%%%%%
%%%%%%%%%%%%%%%%%%%%%%%%%%%%%%%%

\end{document}